\documentclass[12pt]{amsart}
\usepackage{geometry}
\usepackage{graphicx}
\usepackage{amssymb}
\usepackage{epstopdf}
\usepackage{amsmath,amscd}
\usepackage{amsthm}
\usepackage{url,verbatim}
\usepackage{stmaryrd}

\RequirePackage[colorlinks,citecolor=blue,urlcolor=blue]{hyperref}
\usepackage{breakurl}
\theoremstyle{plain}
\newtheorem{theorem}{Theorem}
\newtheorem{definition}[theorem]{Definition}
\newtheorem{lemma}[theorem]{Lemma}
\newtheorem{proposition}[theorem]{Proposition}
\newtheorem{corollary}[theorem]{Corollary}
\newtheorem{example}[theorem]{Example}

\newtheorem{remark}[theorem]{Remark}
\newtheorem{question}[theorem]{Question}

\newcommand\es{\varnothing}
\newcommand\wt{\widetilde}
\newcommand\ol{\overline}
\newcommand\Aut{\mathrm{Aut}}

\newcommand\sH{{\mathcal H}}

\newcommand\HH{{\mathbb H}}
\newcommand\RR{{\mathbb R}}
\newcommand\ZZ{{\mathbb Z}}
\newcommand\NN{{\mathbb N}}
\newcommand\WW{{\mathbb W}}

\newcommand\LL{{\mathbb L}}
\newcommand\TT{{\mathbb T}}
\newcommand\LLp{\LL_+}
\renewcommand\AA{{\mathbb A}}
\newcommand\TLF{\sT}
\newcommand\F{\mathrm{F}}

\renewcommand\a{\alpha}
\newcommand\om{\omega}
\newcommand\g{\gamma}
\newcommand\be{\beta}
\newcommand\si{\sigma}

\newcommand\De{\Delta}
\newcommand\qq{\qquad}
\newcommand\q{\quad}
\newcommand\resp{respectively}

\newcommand\oo{\infty}
\newcommand\sG{{\mathcal G}}

\newcommand\sF{{\mathcal F}}
\newcommand\sP{{\mathcal P}}
\newcommand\sQ{{\mathcal Q}}

\newcommand\sT{{\mathcal T}}
\newcommand\sN{{\mathcal N}}
\newcommand\sS{{\mathcal S}}
\newcommand\sW{{\mathcal W}}

\newcommand\Ga{\Gamma}
\newcommand\Si{\Sigma}
\newcommand\de{\delta}

\newcommand\id{{\bf 1}}

\newcommand\rd{{\mathrm{d}}}
\newcommand\e{\mathrm{e}}
\newcommand\nea{\mathrm{ne}}
\newcommand\nw{\mathrm{nw}}
\newcommand\w{\mathrm{w}}
\newcommand\sw{\mathrm{sw}}
\newcommand\se{\mathrm{se}}

\renewcommand\ell{l}

\newcommand\pd{\partial}

\newcommand\ghf{graph height function}
\newcommand\GHF{group height function}

\newcommand\hdi{$\sH$-difference-invariant}

\newcommand\normal{\trianglelefteq}

\newcommand\wtilde{\widetilde}

\DeclareMathOperator*{\argmax}{argmax}
\newcommand\qua{quadrilateral}
\newcommand\Qh{\sQ_{\mathrm{harm}}}

\newcommand\sgn{{\mathrm {sign}}}
\newcommand\smi{\setminus}
\newcommand\arc[1]{\llbracket #1\rrbracket}

\newcounter{mycount1}\newcounter{mycount2}\newcounter{mycount3}\newcounter{mycount}
\newenvironment{romlist}{\begin{list}{\rm(\roman{mycount1})}%
   {\usecounter{mycount1}\labelwidth=1cm\itemsep 0pt}}{\end{list}}
\newenvironment{romlistp}{\begin{list}{\rm(\roman{mycount1}$'$)}%
   {\usecounter{mycount1}\labelwidth=1cm\itemsep 0pt}}{\end{list}}
\newenvironment{numlist}{\begin{list}{\rm\arabic{mycount2}.}%
   {\usecounter{mycount2}\labelwidth=1cm\itemsep 0pt}}{\end{list}}
\newenvironment{letlist}{\begin{list}{\rm(\alph{mycount3})}%
   {\usecounter{mycount3}\labelwidth=1cm\itemsep 0pt}}{\end{list}}
\newenvironment{Alist}{\begin{list}{\rm\MakeUppercase{\alph{mycount}}.}%
   {\usecounter{mycount}\labelwidth=1cm\itemsep 0pt}}{\end{list}}

\numberwithin{equation}{section}
\numberwithin{theorem}{section}
\numberwithin{figure}{section}

\title{Cubic graphs and the golden mean}
\author{Geoffrey R.\ Grimmett}
\address{Statistical Laboratory, Centre for
Mathematical Sciences, Cambridge University, Wilberforce Road,
Cambridge CB3 0WB, UK} 
\address{School of Mathematics \&\ Statistics, The University of Melbourne, 
Parkville, VIC 3010, Australia}
\email{g.r.grimmett@statslab.cam.ac.uk}
\urladdr{\url{http://www.statslab.cam.ac.uk/~grg/}}
\address{Department of Mathematics,
University of Connecticut,
Storrs, Connecticut 06269-3009, USA} \email{zhongyang.li@uconn.edu}
\urladdr{\url{http://www.math.uconn.edu/~zhongyang/}}

\author{Zhongyang Li}

\begin{document}

\begin{abstract}
The connective constant $\mu(G)$ of a graph $G$ is the exponential growth rate
of the number of self-avoiding walks starting at a given vertex. We investigate
the validity of the inequality $\mu \ge \phi$ for 
infinite, transitive, simple, cubic graphs, where $\phi:= \frac12(1+\sqrt 5)$ is
the golden mean. The inequality is proved for
several families of graphs including (i)  Cayley
graphs  of infinite groups with three generators and strictly positive first Betti number,  (ii)
infinite, transitive, topologically
locally finite (TLF) planar, cubic graphs, and (iii) cubic Cayley graphs with two ends. 
Bounds for $\mu$ are presented for transitive cubic graphs
with girth either $3$ or $4$, and 
for certain quasi-transitive cubic graphs.
\end{abstract}

\date{19 November 2016;  this revision 10 August 2019}

\maketitle

\section{Introduction}\label{sec:int}

Let $G$ be an infinite, transitive, simple, rooted graph, 
and let $\si_n$ be the number of $n$-step self-avoiding walks (SAWs) starting from the root.
It was proved by Hammersley \cite{jmhII} in 1957 that the limit
$\mu=\mu(G):=\lim_{n\to\oo}\si_n^{1/n}$ exists, 
and he called it the \lq connective constant' of $G$. 
A great deal of attention has been devoted to counting SAWs since that introductory mathematics paper, 
and survey accounts of many of the main features of the theory may be found at \cite{bdgs,GL-SAACC, ms}.

A graph is called \emph{cubic} if every vertex has degree $3$, and \emph{transitive} if it is vertex-transitive
(further definitions will be given in Section \ref{sec:def}).
Let $\sG_d$ be the set of infinite, transitive, simple graphs with degree $d$, and
let $\mu(G)$ denote the connective constant of $G \in \sG_d$. 
The letter $\phi$ is used throughout this paper to denote the golden mean $\phi:=\frac12(1+\sqrt 5)$, with numerical value
$1.618\cdots$. The basic question to be investigated here is as follows. 

\begin{question}[\cite{GL-Comb}] \label{qn:big}
Is it the case that $\mu(G) \ge \phi$ for $G \in \sG_3$?
\end{question}

This question has arisen within the study by the current authors
of the properties of connective constants of transitive graphs, see \cite{GL-SAACC} 
and the references therein.  The question is answered affirmatively here for 
certain subsets of $\sG_3$, but we have no complete answer to Question \ref{qn:big}. 
Note that, for $d \ge 4$, 
\begin{equation}\label{eq:lowerbound}
\mu(G)\ge \sqrt{d-1}> \phi, \qq G\in \sG_d,
\end{equation}
by \cite[Thm 1.1]{GL-Comb}.

Here is some motivation for the inequality $\mu(G)\ge \phi$ for $G \in \sG_3$.
It is well known and easily proved that the ladder  $\LL$ (see Figure \ref{fig:2}) has connective
constant $\phi$. Moreover, the number of $n$-step SAWs can be expressed in terms
of the Fibonacci sequence  (an explicit such formula is given in \cite{Zeil}).
It follows that $\mu(G)\ge \phi$ whenever there exists an injection from
a sufficiently large set of rooted $n$-step SAWs on 
$\LL$ to the corresponding set on $G$.
As domain for such injections, we take the set $\WW_n$ of $n$-step \lq eastward'
SAWs on the singly infinite ladder $\LLp$ of Figure \ref{fig:2} (see Section \ref{sec:proof1}).
One of the principal techniques of this article is to construct such injections for
certain families of cubic graphs $G$. 
We state some of our main results next,
and refer the reader to the appropriate sections for the precise terminology
in use.

\begin{theorem}\label{thm:summary}
Let $\sG_3$ be the set of connected, infinite, transitive, cubic graphs.
\begin{Alist}
\item The connective constant $\mu=\mu(G)$ satisfies $\mu\ge \phi$ whenever one 
or more of the following holds:
\begin{letlist}
\item  $G \in \sG_3$ has a transitive \ghf, (Theorem \ref{thm:1}(b)),
\item $G$ is the Cayley graph of the Grigorchuk group with three generators,
(Theorem \ref{thm:Grig}),
\item $G\in \sG_3$ is topologically locally finite, (Theorem \ref{tm93}),
\item $G\in \sG_3$ is the Cayley graph of a finitely presented group
with two ends, (Theorem \ref{thm:2endgp}).
\end{letlist}
Further to the last item, if $\Ga$ is a finitely presented group with infinitely
many ends, it possesses a minimal generator set with Cayley graph $G$
satisfying $\mu(G)\ge \phi$, (Theorem \ref{thm:minimal}).
\item $G \in\sG_3$ satisfies
$$ 
\mu \begin{cases} \in (1.529,1.770) &\text{if $G$ has girth $3$},\\
\in (1.513, 1.900) &\text{if $G$ has girth $4$}.
\end{cases}
$$
(Theorems \ref{g3} and \ref{g4}.)
\end{Alist}
\end{theorem}  

There are many infinite, transitive, cubic graphs, and we are unaware of a complete taxonomy.
Various examples and constructions are described in Section \ref{ex:ex}
(including the illustrious case of the hexagonal lattice, see \cite{ds}), and 
the inequality $\mu\ge\phi$ is discussed in each case. 
In our search for cubic graphs,
no counterexample has been knowingly revealed. 
Our arguments can frequently be refined to obtain stronger lower bounds for connective constants than $\phi$,
but we do not explore that here.

A substantial family of cubic graphs arises through the application of
the so-called \lq Fisher transformation' to a $d$-regular graph (see
Section \ref{sec:girth}). 
We make explicit mention of the Fisher transformation here since
it provides a useful technique in the study of connective constants.  

The family of Cayley graphs provides a set of transitive graphs of special interest and structure.
The Cayley graph of the Grigorchuk group
is studied by a tailored argument in Section \ref{sec:Grig}.
In Section \ref{sec:ends}, we treat $2$-ended Cayley graphs   
and certain $\oo$-ended Cayley graphs.

We make a final note concerning the numbers $\phi$ and $\sqrt 2$, in their roles in  
Question \ref{qn:big}, and in \eqref{eq:lowerbound} with $d=3$.
Bucher and Talambutsa \cite{BT1,BT2} have derived lower bounds and equalities for 
exponential growth rates of non-trivial free and amalgamated products. 
In particular, they show
there is a gap between $\sqrt 2$ and $\phi$ for the growth rates 
of free products.
They are able to study the infimum growth rate
of free groups over all  generator sets. 
It is elementary that the infimum
growth rate is a lower bound for the connective constants
of the corresponding Cayley graphs (studied in \cite{Gilch}).

This paper is structured as follows.
General criteria that imply $\mu\ge\phi$ are presented in Section \ref{sec:ineq}
and proved in Section \ref{sec:proof1}.
In Section \ref{ex:ex} is given a list of cubic graphs known to
satisfy $\mu\ge \phi$
(for some such graphs, the inequality follows from earlier
results as noted, and for others by the results of the current article). 
So-called \emph{transitive} graph height functions are discussed
in Section \ref{sec:ghf}, including sufficient conditions for their existence.
Upper and lower bounds for connective constants for cubic graphs with
girth $3$ or $4$ are stated and proved in Section \ref{sec:girth}. 
The Grigorchuk group  is considered in Section \ref{sec:Grig}.
In Section \ref{sec:tlf}, it is proved that $\mu\ge \phi$ for
all transitive, topologically locally finite (TLF) planar, cubic graphs. The final Section \ref{sec:ends}
is devoted to $2$- and $\oo$-ended Cayley graphs.
(Theorem \ref{thm:Grig} and much of the proof of Theorem \ref{thm:2endgp}
are due to Anton Malyshev (personal communication).)

\section{Preliminaries}\label{sec:def}

The graphs $G=(V,E)$ of this paper will be assumed to be connected, infinite, and simple
(parallel edges will make a brief appearance in and around Proposition \ref{lem:f4}).
We write $u \sim v$ if $\langle u,v\rangle \in E$, and say that $u$ and $v$ are \emph{neighbours}.
The set of neighbours of $v\in V$ is denoted $\pd v$.
The \emph{degree} $\deg(v)$ of vertex $v$ is the number of edges
incident to $v$, and $G$ is called \emph{cubic} if $\deg(v)=3$ for $v\in V$.  

The automorphism group of  $G$ is
written $\Aut(G)$. A subgroup $\Ga \le \Aut(G)$ is said to 
\emph{act transitively} 
if, for $v,w\in V$, there exists $\g \in \Ga$ with $\g v=w$.
It   \emph{acts quasi-transitively} if there is a finite
subset $W\subseteq V$ such that, for $v \in V$, there exist
$w \in W$ and $\g \in \Ga$ with $\g v =w$.
The graph is called \emph{(vertex-)transitive} 
(\resp, \emph{quasi-transitive}) if $\Aut(G)$ acts transitively
(\resp, quasi-transitively). 

A \emph{walk} $w$ on the (simple) graph $G$ is
a sequence $w=(w_0,w_1,\dots, w_n)$ of vertices $w_i$ such that $n \ge 0$ and
$e_i=\langle w_i, w_{i+1}\rangle\in E$ for $i \ge 0$. The \emph{length} $|w|$ of a walk $w$ 
is the number of its edges, and $w$ is called \emph{closed}
if $w_0=w_n$. The distance $d_G(v,w)$ between vertices $v$, 
$w$ is the length of the shortest walk of $G$ between them.

An \emph{$n$-step self-avoiding walk} (SAW) 
on $G$ is  a walk $w=(w_0,w_1,\dots, w_n)$ of length $n\ge 0$ with no repeated vertices.
The walk $w$ is
called \emph{non-backtracking} if $w_{i+1}\ne w_{i-1}$ for $i \ge 1$.
A \emph{cycle} is a walk $(w_0,w_1,\dots, w_n)$
with $n \ge 3$ such that $w_0=w_n$, and $w_i \ne w_j$ for $0 \le i < j < n$.
Note that a cycle has a specified orientation.
The \emph{girth} of $G$ is the length of its shortest cycle.
A \emph{triangle} (\resp, \emph{\qua}) is a 
cycle of length $3$ (\resp, $4$).

We denote by $\sG$ (\resp, $\sQ$) the set of infinite, rooted, connected, transitive (\resp, quasi-transitive), 
simple graphs 
with finite vertex-degrees.
The subset of $\sG$ (\resp, $\sQ$)  containing graphs with degree $d$ is denoted $\sG_d$
(\resp, $\sQ_d$), and the subset of $\sG_d$ (\resp, $\sQ_d$)  containing graphs with girth $g$ is denoted 
$\sG_{d,g}$ (\resp, $\sQ_{d,g}$) .
The root of such graphs is denoted $0$ (or $\id$ when the graph is a Cayley graph of a group
with identity $\id$).

Let $\Si_n(v)$ be the set of $n$-step SAWs starting at $v\in V$, and
$\si_n(v):=|\Si_n(v)|$ its cardinality.
Let $G\in \sQ$.
It is proved in \cite{jmhII,hm}  that the limit
\begin{equation}\label{def:cc}
\mu=\mu(G) :=\lim_{n\to\oo} \si_n(v)^{1/n}, \qq v \in V,
\end{equation}
exists, and $\mu(G)$ is called the \emph{connective constant} of $G$.
We shall have use for the SAW \emph{generating function}
\begin{equation}\label{eq:genfn}
Z_v(\zeta) = \sum_{\substack{\pi\text{ a SAW}\\\text{from $v$}}} \zeta^{|\pi|}
= \sum_{n=0}^\oo \si_n(v) \zeta^n,
\qq v \in V,\ \zeta \in \RR.
\end{equation}
By \eqref{def:cc}, each $Z_v$ has radius of convergence $1/\mu(G)$.
We shall sometimes consider SAWs joining midpoints of edges of $G$
(in the manner of \cite{ds,GrL2}).

There are two (related) types of graph functions relevant to this work.
We recall first the definition of a \lq\ghf', as introduced in \cite{GL-loc} in the context
of connective constants.

\begin{definition}[\cite{GL-loc}] \label{def:height}
Let $G \in \sQ$. A \emph{\ghf} on $G$ is a pair $(h,\sH)$ such that:
\begin{letlist}
\item $h:V \to \ZZ$ and $h(0)=0$, 
\item $\sH$ is a subgroup of $\Aut(G)$ acting quasi-transitively on $G$ 
such that $h$ is \emph{\hdi} in the sense that
$$
h(\a v) - h(\a u) = h(v) - h(u), \qq \a \in \sH,\ u,v \in V,
$$
\item for  $v\in V$, there exist $u,w \in \pd v$ such that
$h(u) < h(v) < h(w)$.
\end{letlist}
A \ghf\ $(h,\sH)$ of $G$ is called \emph{transitive} if $\sH$ acts transitively on $G$.
\end{definition}

The properties of normality and unimodularity of the group $\sH$ 
are discussed in \cite{GL-loc}, but do not appear to be especially relevant 
to the current work.

Secondly, we remind the reader of
the definition of a harmonic function on a graph $G=(V,E)$.
A function $h:V \to\RR$ is called \emph{harmonic} if
$$
h(v) = \frac1{\deg(v)}\sum_{u \sim v} h(u), \qq v \in V.
$$

Cayley graphs of finitely generated groups (with symmetric generator sets) make appearances in this
paper, and the reader is referred to \cite{GL-Cayley,  GL-amen} for background 
material on such graphs.  We denote by $\id$ the identity of any group $\Ga$
under consideration.

\section{General results}\label{sec:ineq}

Let $G=(V,E)$ be an infinite, connected graph. For $h:V\to\RR$,
we define two functions
$m:V\to V$ and $M:V\to\RR$, depending on $h$, by
\begin{equation}\label{eq:0}
m(u)\in \argmax\{h(x)-h(u): x \sim u\}, \q
M_u=h(m(u))-h(u), \qq u \in V.
\end{equation}
There may be more than one candidate vertex for $m(u)$, and hence more than
one possible value for the term $M_{m(u)}$. 
If so, we make a choice for the value $m(u)$, and we fix $m(u)$ thereafter.
Let $q(v)$ denote the unique neighbour of $v:=m(u)$ other than $u$ and $m(v)$.
We shall apply the functions $m$ and $q$ repeatedly, and shall omit parentheses in that,  
for example, $mqm(u)$ denotes the vertex $m(q(m(u)))$,
and $(qm)^2(u)$ denotes $qmqm(u)$.
This notation is illustrated in Figure \ref{fig:1}. 

\begin{figure}
\centerline{\includegraphics[width=0.4\textwidth]{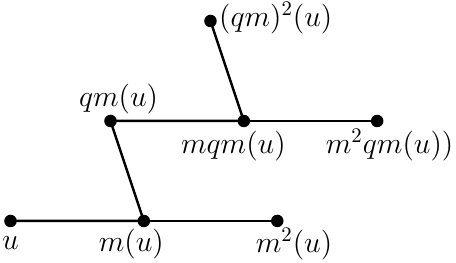}}
\caption{An illustration of the notation of equations \eqref{eq:1}--\eqref{eq:2}.}
\label{fig:1}
\end{figure}

Let $\Qh\subseteq \sQ_3$
be the subset of graphs $G$ with the following properties:
there exists $h:V\to\RR$ such that $h$ is harmonic and,
for $u \in V$,
\begin{align}\label{eq:1}
M_{m(u)}-M_u &< \min\{M_u,M_{qm(u)}\}, \\
2M_{qm(u) }&> M_{m(u)}-M_u +M_{mqm(u)}.\label{eq:2}
\end{align}
Although inequalities \eqref{eq:1} and \eqref{eq:2} lack obvious motivation,
they turn out to be useful (see Theorem \ref{thm:1}) in 
establishing certain cases of the inequality $\mu(G)\ge \phi$. 
We note two consequences of \eqref{eq:1} and \eqref{eq:2}:
\begin{letlist}
\item since $h$ is assumed harmonic, we have $M_u \ge 0$ for $u \in V$, and hence
$M_u>0$ by \eqref{eq:1},
\item it is proved at \eqref{eq:1new} that, subject to \eqref{eq:1} and \eqref{eq:2},
$$h(qm(u)) >h(u),\ h(mqm(u)) > h(m(u)),\ h((qm)^2(u)) > h(m(u)),
$$ 
whence $qm(u)\ne u$, $mqm(u)\ne m(u)$, $(qm)^2(u) \ne m(u)$.
\end{letlist}

Conditions \eqref{eq:1}--\eqref{eq:2} will be used in the proof of part (a) of the following theorem. 
Less obscure but still sufficient conditions are contained in Remark \ref{rem:lesser}, following.

\begin{theorem}\label{thm:1}
We have that $\mu(G) \ge \phi$ if any of the following hold.
\begin{letlist}
\item $G\in \Qh$.
\item $G\in\sG_3$ has a transitive \ghf.
\item $G\in\sQ_{3,g}$ where $g \ge 3$, and
there exists a function $h: V(G)\rightarrow\RR$ 
such that, for  $u\in V$,
\begin{gather}\label{qm}
h(qm(u))>h(u),\q h(mqm(u))>h(m(u)),\\
h((qm)^\g q(u))>h(u),\label{ag}
\end{gather}
where $\g=\lceil \frac12(g-1)\rceil$.
\item $G\in\sQ_{3,g}$ where $g \ge 3$, 
and there exists a harmonic function $h$ on $G$ 
satisfying \eqref{eq:1} and \eqref{ag}.
\end{letlist}
\end{theorem} 
 
\begin{remark}\label{rem:lesser}
Condition \eqref{eq:1} holds whenever there exists  $A>0$ and 
a harmonic function $h:V\to\RR$ such that, for $u \in V$,  $A< M_u \le 2A$.
Similarly, both \eqref{eq:1} and \eqref{eq:2} hold whenever there exists  $A>0$
such that, for $u \in V$,  $2A< M_u \le 3A$.
\end{remark}

Let $\Ga$ be an infinite, finitely presented group, and let $G$ be a locally finite 
Cayley graph of $\Ga$. If there exists a surjective homomorphism 
$F$ from $\Ga$ to $\ZZ$, then $F$ is a transitive \ghf\ on $G$ (see \cite{GL-Cayley}).
Such a \ghf\ is called a \emph{\GHF}.

\begin{example}\label{examples}
Here are some examples of Theorem \ref{thm:1} in action.
\begin{letlist}
\item The hexagonal lattice $\HH$ supports a harmonic function $h$ with $M_u \equiv 1$,
so that part (a) of the theorem applies (see Remark \ref{rem:lesser}). To see this,
we embed $\HH$ into the plane as in the dashed lines of Figure \ref{4612ed}.
Let each edge have length $1$, and let $h(u)$ be the horizontal coordinate of the vertex $u$.
(The exact value of $\mu(\HH)$ was proved in \cite{ds}.)
\item The Cayley graph of a finitely presented group $\Ga=\langle S\mid R\rangle$
with $|S|=3$ has a transitive \ghf\ whenever it has a \GHF, and hence part (b) applies.
See Theorem \ref{qtt} for a sufficient condition on a transitive cubic graph
to possess a transitive \ghf.
\item The Archimedean lattice $\AA=\arc{4,6,12}$ 
lies in $\sQ_{3,4}$ and possesses a harmonic 
function satisfying \eqref{eq:1}
and \eqref{eq:2}. 
The harmonic function in question is illustrated in Figure \ref{fig:arch},
and the claimed inequalities may be checked from the figure.
See also Remark \ref{rem:arch-hex2}.
\item The inequality $\mu(\AA)\ge\phi$ may be proved also as follows. 
The lattice $\AA$ can be embedded into the plane as in the solid lines of Figure \ref{4612ed}.
As in (a) above, let $h(u)$ be the horizontal coordinate of $u$. 
By Theorem \ref{thm:1}(c), the connective constant  is at least $\phi$.
\end{letlist}
\end{example}

\begin{figure}
\centerline{\includegraphics[width=0.7\textwidth]{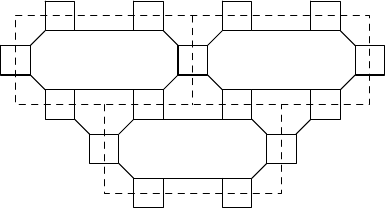}}
\caption{The dashed lines form the hexagonal lattice $\HH$, 
and the solid lines  the $\arc{4,6,12}$ lattice}  
\label{4612ed}
\end{figure}

The proof of Theorem \ref{thm:1} is found in Section \ref{sec:proof1}.

\section{Examples of infinite, transitive, cubic graphs}\label{ex:ex}

\subsection{Cubic graphs with $\mu\ge\phi$}\label{sec:exs}

We list here examples of infinite, cubic graphs $G$ with $\mu(G) \ge \phi$.
As mentioned earlier, we have no example that violates the inequality
(however, see Section \ref{sec:oq}). 
A number of these examples are well known, and 
others have been studied by other authors.
In some cases, Theorem \ref{thm:1} may be applied,
and such cases are prefixed by
the part of the theorem that applies. Most of these examples
are transitive, and all are quasi-transitive.

\begin{Alist}
\item (b) The  $3$-regular tree has connective constant $2$.

\item (a) The \lq ladder'  $\LL$ (see Figure \ref{fig:2})
has $\mu=\phi$. This exact value is elementary
and well known; see, for example, \cite[p.\ 284]{GL-Comb}.

\item\label{twistedladder} 
The \lq twisted ladder' $\TT\LL$ (see Figure \ref{fig:rotL}) has $\mu=
\sqrt{1+\sqrt 3}) \approx 1.653 > \phi$. To see this, observe that the
generating function of SAWs from $0$ (see \eqref{eq:genfn}) that move only 
eastwards or within \qua s is $Z(\zeta)=\sum_{m=0}^\oo  f(\zeta)^m$,
where $f(\zeta) = 2\zeta^2+2\zeta^4$. The radius of convergence, $1/\mu(\TT\LL)$, of $Z$ is the root of the equation $f(\zeta)=1$.

\item (a) 
The hexagonal lattice $\HH$ satisfies $\mu(\HH)\ge\phi$, by Example \ref{examples}(a).
It has been proved in \cite{ds} that  $\mu=\sqrt{2+\sqrt 2}$. 

\item (a) It is explained in \cite[Ex.\ 4.2]{GrL3} that the square/octagon lattice 
$\arc{4,8,8}$
satisfies $\mu > \phi$.

\item  (a,\,c) 
The Archimedean $\arc{4,6,12}$ lattice has connective constant at least
$\phi$. See Example \ref{examples}(c,\,d) and Remark \ref{rem:arch-hex2}.

\item (b) The Cayley graph of the lamplighter group has
a so-called \GHF, and hence a transitive \ghf. See 
Example \ref{examples}(b) and \cite[Ex.\ 5.3]{GL-Cayley}.

\item
The following examples concern so-called Fisher graphs 
(see \cite{GrL2} and Section \ref{sec:girth}). 
For $G \in \sG_3$, the \emph{Fisher graph} $G_\F$  ($\in\sQ_3$) is obtained by  
replacing each vertex by a triangle.  
It is shown at \cite[Thm 1]{GrL2} that the value of $\mu(G_\F)$ may be deduced 
from that of $\mu(G)$, and furthermore that  $\mu(G_\F) > \phi$
whenever $\mu(G) > \phi$.

\item\label{31212} In particular, the Fisher graph $\HH_\F$ 
of $\HH$ satisfies $\mu(\HH_\F) > \phi$.

\item The Archimedean lattices mentioned above are 
the hexagonal lattice $\HH=\arc{6,6,6}$, the square/octagon lattice $\arc{4,8,8}$,
together with $\arc{4,6,12}$, and $\HH_\F=\arc{3,12,12}$.
To this list we may add the ladder  $\LL=\arc{4,4,\oo}$.

These are examples of so-called transitive, TLF-planar graphs \cite{DR09}, and all such
graphs are shown in Section \ref{sec:tlf} to satisfy $\mu\ge\phi$.

\item More generally, if $G \in \sG_d$ where $d \ge 3$, and 
$$
\frac1{\mu(G)} \le 
\begin{cases} \dfrac1{\phi^{r+1}} + \dfrac1{\phi^{r+2}}&\text{if } d=2r+1,\\
\dfrac2{\phi^{r+1}} &\text{if } d=2r,
\end{cases}
$$
then its (generalized) Fisher graph satisfies $\mu(G_\F)\ge \phi$.
See Proposition \ref{lem:f4}.

Since $\mu\le d-1$, the above display can be satisfied only if $d\le 10$.

\item The Cayley graph $G$ of the group 
$\Gamma=\langle S\mid R\rangle$, where $S=\{a,b,c\}$ and $R=\{c^2,ab,a^3\}$, 
is the Fisher graph of the $3$-regular tree, and hence $\mu(G) > \phi$.
The exact value of $\mu(G)$ may be calculated by \cite[Thm 1]{GrL2}
(see also Proposition \ref{lem:f4}(a) and \cite[Ex.\ 5.1]{Gilch}).

We note that the $\arc{3,12,12}$ lattice
is a quotient graph of $G$ by adding the further relator $(ac)^6$.
Since the last lattice has connective constant at least $\phi$, so does $G$
(see \cite[Cor.\ 4.1]{GrL3}).

\item The Cayley graph $G$ of the group 
$\Gamma=\langle S\mid R\rangle$, where $S=\{a,b,c\}$ and $R=\{a^2,b^2,c^2,(ac)^2\}$,
is the generalized Fisher graph of the $4$-regular tree. 
The connective constant $\mu(G)$ may be calculated exactly, as in Theorem \ref{g5},
and satisfies $\mu>\phi$.

Since the ladder  $\LL$ is the quotient graph of $G$ obtained by 
adding the further relator $(bc)^2$, we have by \cite[Cor.\ 4.1]{GrL3} that $\mu(G)>\phi$.
(see \cite{GrL3}).

\item The Cayley graph of the Grigorchuk group with three generators has $\mu\ge \phi$. The proof
uses a special construction 
due to Malyshev
based on the orbital Schreier graphs, and is presented in Section \ref{sec:Grig}.

\item (b)
A \GHF\ of a Cayley graph is also a transitive \ghf\ (see \cite{GL-Cayley}). 
Therefore, any cubic Cayley graph 
with a \GHF\  satisfies $\mu\ge \phi$.

\item (b)
Let $G\in \sG_3$ be such that: there exists $\sH \le \Aut(G)$ that acts transitively but is not unimodular. 
By \cite[Thm 3.5]{GL-Cayley}, $T$ has a transitive \ghf,
whence $\mu\ge \phi$.

\end{Alist}

\subsection{Open question}\label{sec:oq}

We mention a general situation in which we are unable to
show that $\mu\ge\phi$.
 Let $G$ be the Cayley graph of an infinite, finitely generated, 
virtually abelian group $\Ga=\langle S\mid R\rangle$ with $|S|=3$. 
Is it generally true that $\mu(G)\ge \phi$?
Whereas such groups are abelian-by-finite, the finite-by-abelian case is
fairly immediate (see Theorem \ref{thm:exact}).

A method for constructing such graphs was 
described by Biggs \cite[Sect.\ 19]{Biggs} and developed by Seifter \cite[Thm 2.2]{NS}. 
Cayley graphs with two or more ends are considered in Section \ref{sec:ends}.

\section{Proof of Theorem \ref{thm:1}}
\label{sec:proof1}

We begin with some notation that will be used throughout this article.
Let $\LLp$ be the singly-infinite ladder  of Figure \ref{fig:2}. An
\emph{eastward} SAW on $\LLp$
is a SAW  starting at $0$ that, at each stage, steps either to the right 
(that is, \emph{horizontally}, denoted H) or between 
layers (that is, \emph{vertically}, denoted V). 
Note that the first step of an eastward walk is necessarily H,
and every V step is followed by an H step.
Let $\WW_n$ be the set of $n$-step eastward SAWs on $\LLp$.  

\begin{figure}
\centerline{\includegraphics[width=0.5\textwidth]{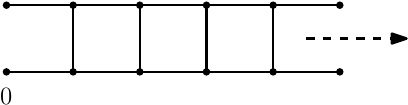}}
\caption{The singly infinite ladder  $\LLp$. 
The doubly infinite ladder $\LL$
extends to infinity both leftwards and rightwards.}
\label{fig:2}
\end{figure}

It is clear that $\WW_n$ is in one--one correspondence 
with the set of $n$-letter words $w$ in the alphabet $\{$H,\,V$\}$
that start with the letter H and have
no pair of consecutive appearances of the letter V.
We shall frequently consider $\WW_n$ as this set of words, and we shall make use of 
the set $\WW_n$ throughout this paper.

It is elementary, by considering the first two steps,  
that $\eta_n=|\WW_n|$ satisfies the recursion
$$
\eta_n = \eta_{n-1}+\eta_{n-2}, \qq n \ge 2,
$$
with $\eta_0=\eta_1=1$. Therefore,
\begin{equation}\label{eq:3}
\lim_{n\to\oo} \eta_n^{1/n} = \phi.
\end{equation}

\begin{figure}
\centerline{\includegraphics[width=0.6\textwidth]{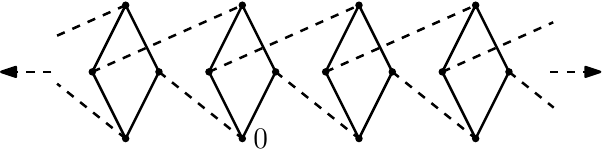}}
\caption{The doubly infinite 
\lq twisted ladder' $\TT\LL$ is obtained from the ladder by twisting
every other \qua.}
\label{fig:rotL}
\end{figure}

Let $G =(V,E)\in\sQ_3$, and 
let $\sW_n$ denote the set of $n$-step walks starting at the root $0$.
Let $h:V \to\RR$.
We shall first construct an injection $f:\WW_n \to \sW_n$,
and then we will show that, subject to appropriate conditions,
each $f(w)$ is a SAW. 
In advance of giving the formal definition of $f$,
we explain it informally. When thinking of an element of $\WW_n$ as a word of 
length $n$, we apply the function $m$ at every appearance of H, 
and $q$ at every appearance of V; for example, the word HVHH corresponds to 
the vertex $m^2qm(0)$.

\begin{definition}\label{saw-def}
For $w=(w_1w_2\cdots w_n) \in\WW_n$, we
let $f(w)=(f_0,f_1,\dots, f_n)$ be the $n$-step walk on $G$ given as follows.\begin{numlist}
\item $f_0=0$, $f_1=m(f_0)$.
\item Assume $k \ge 1$ and $(f_0,f_1,\dots,f_k)$ have been defined.
\begin{letlist} 
\item If $w_{k+1}=$ {\rm H}, then $f_{k+1}=m(f_k)$.
\item If $w_{k+1}=$ {\rm V}, then $f_{k+1}=q(f_k)$.
\end{letlist}
\end{numlist}
\end{definition}

\begin{lemma}\label{lem-1}
The function $f$ is an injection from $\WW_n$ to $\sW_n$.
\end{lemma}

\begin{proof}
Let $w,w'\in \WW_n$ satisfy $w\ne w'$, 
and let $l$ be such that $w_i=w'_i$ for $1\le i<l$, and $w_l=$ H, $w_l'=$ V.
It is necessarily the case that $l \ge 2$ and $w_{l-1}=w'_{l-1}=$ H.
We have that $f_i(w)=f_i(w')$ for $1\le i < l$, and
\begin{equation*}
f_{l}(w)=m^2(u),\qq f_{l}(w')=qm(u),
\end{equation*}
where $u=f_{l-2}(w)$. 
Since $m^2(u)\neq qm(u)$, we have $f(w)\neq f(w')$ as required.
\end{proof}

\begin{proof}[Proof of Theorem \ref{thm:1}(a)]
Let $G =(V,E)\in \Qh$, and let $h:V \to\RR$ be harmonic such that
\eqref{eq:1}--\eqref{eq:2} hold. 

\begin{lemma}\label{lem1}
The function $f$ is an injection from $\WW_n$ to $\Si_n(0)$.
\end{lemma}

\begin{proof}
In the light of Lemma \ref{lem-1}, it suffices to show that each $f(w)$ is a SAW.

Let $u \in V$. The three neighbours of $m(u)$ are $u$, $qm(u)$, $m^2(u)$ (see Figure \ref{fig:1}).
Since $h$ is harmonic,
$$
3h(m(u)) = h(u)+h(qm(u))+h(m^2(u)), \qq u \in V,
$$
so that
\begin{equation}\label{eq:harm}
h(qm(u))-h(m(u)) = M_u-M_{m(u)}.
\end{equation}
Therefore, by \eqref{eq:1},
\begin{align}\label{eq:5a}
h(qm(u))-h(u) &= M_u-M_{m(u)} +\bigl[h(m(u))-h(u)\bigr]\\
&=2M_u-M_{m(u)} > 0, \nonumber\\
h(mqm(u))-h(m(u))&= M_u-M_{m(u)}+\bigl[h(mqm(u))-h(qm(u)\bigr]\label{eq:5b}\\
&  =M_u-M_{m(u)}+M_{qm(u))}  >0,\nonumber
\end{align}
and, by \eqref{eq:harm} with $u$ replaced by $qm(u)$, and \eqref{eq:2},
\begin{align}\label{eq:6}
h((qm)^2(u)) -h(m(u)) &= M_u-M_{m(u)} + M_{qm(u)}+\bigl(M_{qm(u)}-M_{mqm(u)}\bigr)\\
&  >0. \nonumber
\end{align}
See Figure \ref{fig:1} again.
By \eqref{eq:5a}--\eqref{eq:6},
\begin{equation}\label{eq:1new}
qm(u) \ne u,\qq mqm(u)\ne m(u), \qq (qm)^2(u) \ne m(u),
\end{equation}
as claimed above Theorem \ref{thm:1}.

Let $w \in \WW_n$.
Let $S_k$ be the statement that
\begin{letlist}
\item $f_0,f_1,\dots, f_k$ are distinct, and
\item if $w_k=$ H, then $h(f_{k})>h(f_i)$ for $0\le i \le k-1$, and
\item if $w_k=$ V, then $h(f_{k}) > h(f_i)$ for $0\le i \le k-2$.
\end{letlist}
If $S_k$ holds for every $k$, then the $f_k$ are distinct, whence $f(w)$ is a SAW.
We shall prove the $S_k$ by induction.

Evidently, $S_0$ and $S_1$ hold.
Let $K\ge 2$ be such that $S_k$ holds for $k<K$, and consider $S_K$. 

\begin{numlist}
\item
Suppose first that $w_K=$ V, so that $w_{K-1}=$ H. By \eqref{eq:5a} (or \eqref{eq:1new})
with $u=f_{K-2}$ and $v=m(f_{K-2})=f_{K-1}$, we have that $h(f_K) > h(f_{K-2})$.
\begin{letlist}
\item If $w_{K-2}=$ H, the claim follows by $S_{K-2}$.
\item Assume $w_{K-2}=$ V (so that, in particular, $K \ge 4$).
We need also to show that $h(f_K)>h(f_{K-3})$. 
In this case, we take $u=f_{K-4}$ so that $m(u)=m(f_{K-4})=f_{K-3}$,
and $(qm)^2(u)=f_{K}$  in \eqref{eq:6},
thereby obtaining that $h(f_K)>h(f_{K-3})$ as required.
\end{letlist}

\item
Assume next that $w_K=$ H.
\begin{letlist}
\item If $w_{K-1}=$ H, the relevant claims of $S_K$ follow by $S_{K-1}$ and the fact
that $f_K = m(f_{K-1})$. 
\item If $w_{K-1}=$ V, then $w_{K-2}=$ H. By \eqref{eq:5b}, $h(f_K)>h(f_{K-2})$,
and the claim follows by $S_{K-1}$ and $S_{K-2}$.
\end{letlist}
\end{numlist}
This completes the induction, and the lemma is proved.
\end{proof}

By Lemma \ref{lem1}, $|\Si_n(0)| \ge |\WW_n|$, and part (a) follows by \eqref{eq:3}.
\end{proof}

\begin{proof}[Proof of Theorem \ref{thm:1}(b)]
Let $G\in\sG_3$ and let $(h,\sH)$ be a transitive \ghf. 
For $u\in V$, let
$M=\max\{h(v)-h(u): v \sim u\}$ as in \eqref{eq:0}. We have that 
$M>0$ and, by transitivity, $M$ does not depend
on the choice of $u$. 
Since $h$ is \hdi, the neighbours of any
$v\in V$ may be listed as $v_1$, $v_2$, $v_3$ where 
\begin{equation*}
h(v_i)-h(v) = \begin{cases} M &\text{ if } i=1,\\
-M &\text{ if } i=2,\\
\eta &\text{ if } i=3,
\end{cases}
\end{equation*}
where $\eta$ is a constant satisfying $|\eta| \le M$.
By the transitive action of $\sH$, we have that $-\eta\in\{-M,\eta,M\}$, whence
$\eta\in\{-M,0,M\}$.

If $\eta=0$, $h$ is harmonic and satisfies \eqref{eq:1}--\eqref{eq:2}, and the claim
follows by part (a).  If $\eta=M$, it is easily seen that the construction of Definition \ref{saw-def}
results in an injection from $\WW_n$ 
to $\Si_n(v)$. If $\eta=-M$, we replace $h$ by $-h$
to obtain the same conclusion.
\end{proof}

\begin{proof}[Proof of Theorem \ref{thm:1}(c)]
This is a variant of the proof of part (a). 
With $\g=\lceil\frac12(g-1)\rceil$ as in the theorem,   let $T_k$ be the statement that:
\begin{letlist}
\item if $w_k=$ H, 
then $h(f_{k})>h(f_i)$ for $0\le i \le k-1$, and
\item if $w_k =$ V, 
then $h(f_{k}) > h(f_i)$ whenever $i$ satisfies either
\begin{romlist}
\item $i=k-2s\ge 0$ for $s \in \NN$, or
\item $i=k-(2t+1)\ge 0$ for $t\in\NN$, $t\geq \g$.
\end{romlist}
\end{letlist}

\begin{lemma}\label{lem:10}
Assume  that $T_k$ holds for every $k$. The vertices $f_k$ are distinct, 
so that each $f(w)$ is a SAW.
\end{lemma}

Part (c) follows from this by Lemma \ref{lem-1}, as in the proof of part (a).

\begin{proof}[Proof of Lemma \ref{lem:10}]
Let $k \ge 1$.
If $w_k=$ H then, by $T_k$, $f_k \ne f_0,f_1,\dots, f_{k-1}$. Assume that
$w_k = $ V. By $T_k$, we have that $f_k \ne f_i$ for $0\le i<k$ except possibly for the values
$i\in I:=\{k-1, k-3, \dots, k-(2\g-1)\}$.
If $f_k=f_i$ with $i\in I$, then $G$ has girth not exceeding $2\g-1$ ($<g$), a contradiction.
Since this holds for all $k$, the $f_k$ are distinct, and hence $f(w)$
is a SAW.
\end{proof}

We next prove the $T_k$ by induction. Evidently, $T_0$ and $T_1$ hold.
Let $K\ge 2$ be such that $T_k$ holds for $k<K$, and consider $T_K$. 

\emph{Suppose first that} $w_K=$ V, so that $w_{K-1}=$ H. By \eqref{qm} 
with $u=f_{K-2}$,
\begin{equation}
h(f_K) > h(f_{K-2}).\label{kk21}
\end{equation}
\begin{Alist}
\item Assume $w_{K-2}=$ H.
By \eqref{kk21} and $T_{K-2}$, we have that $h(f_K)>h(f_i)$ for $i\leq K-2$.
\item Assume $w_{K-2}=$ V (so that, in particular, $K \ge 4$). By $T_{K-2}$,
\begin{alignat*}{2}
h(f_{K-2})&> h(f_{K-2-2s}),\ &&\text{for}\  s\in\NN, \ K-2-2s\geq 0,\\
h(f_{K-2})&> h(f_{K-2-(2t+1)}),\q&& \text{for}\ t\geq \g,\ K-2-(2t+1)\geq 0.
\end{alignat*}
Hence, by \eqref{kk21},
\begin{alignat}{2}
h(f_{K})&> h(f_{K-2s}),\ &&\text{for}\ s\in\NN,\ K-2s\geq 0,\label{kge}\\
h(f_{K})&> h(f_{K-(2t+3)}),\q&& \text{for}\ t\geq \g,\ K-2-(2t+1)\geq 0.\notag
\end{alignat}
It remains to show that 
\begin{equation}\label{kk6}
h(f_K)>h(f_{K-(2\g+1)}).
\end{equation}  
Exactly one of the following two cases  occurs.
\begin{romlist}
\item There are two (or more) consecutive appearances of H in $w_K,\dots,w_{K-2\g}$. 
In this case there exists $1\leq t\leq \g$ such that $w_{K-2t}=$ H, implying 
by $T_{K-2t}$ that
\begin{equation*}
h(f_{K-2t})>h(f_i), \qq 0\leq i \le K-{2t}-1.
\end{equation*}
Inequality \eqref{kk6} follows by \eqref{kge}.
\item We have that $(w_K,\dots,w_{K-2\g})= ($V,H,V,H,$\dots$,V$)$,
in which case \eqref{kk6} follows from  \eqref{ag}.
\end{romlist}
\end{Alist}

\emph{Suppose next that} $w_K=$ H.
\begin{Alist}
\item If $w_{K-1}=$ H, the relevant claims of $T_K$ follow by $T_{K-1}$ and the fact
that $f_K = m(f_{K-1})$. 
\item If $w_{K-1}=$ V, then $w_{K-2}=$ H. By \eqref{qm} and $T_{K-2}$, 
$h(f_K)>h(f_{K-2})>h(f_i)$ for $0\le i \le K-3$.
Finally, $h(f_K)>h(f_{K-1})$ since $f_K = m(f_{K-1})$.
\end{Alist}
This completes the induction.
\end{proof}

\begin{proof}[Proof of Theorem \ref{thm:1}(d)]
It suffices by part (c) to show that the harmonic function $h$ satisfies \eqref{qm}.
This holds as in \eqref{eq:5a} and \eqref{eq:5b}.
\end{proof}

\section{Transitive graph height functions}\label{sec:ghf}

By Theorem \ref{thm:1}(b), the possession of a transitive \ghf\ 
suffices for the inequality $\mu(G)\ge\phi$.
It is not currently known exactly which $G \in \sG_3$ possess  transitive \ghf s, and it is shown
in \cite[Thms 5.1, 8.1]{GL-amen} that the Cayley graph of neither the Grigorchuk group 
nor the Higman group has a \ghf\ at all.
We pose a weaker question here. 
Suppose $G \in\sG_3$ possesses a \ghf\ $(h,\sH)$. Under what
further condition does $G$ possess a \emph{transitive} \ghf? A natural 
candidate function $g:V\to\ZZ$ is obtained as follows. 

\begin{proposition}\label{prop1}
Let $\Ga$ act transitively on $G=(V,E)\in \sG_d$ where $d \ge 3$.
Assume that $(h,\sH)$ is a \ghf\ of $G$, where $\sH \normal \Ga$ and $[\Ga:\sH]<\oo$.
Let $\kappa_i\in \Ga$ be representatives of the cosets, so that $\Ga/\sH=\{\kappa_i\sH: i \in I\}$, and let
\begin{equation}\label{eq:7}
g(v) = \sum_{i\in I} h(\kappa_i v), \qq v \in V.
\end{equation}
The function $g:V\to\ZZ$ is $\Ga$-difference-invariant.
\end{proposition}

A variant of the above will be useful in the proof of Theorem \ref{thm:2endgp}.

\begin{proof}
The function $g$ is given in terms of the representatives $\kappa_i$ of the cosets, but its differences
$g(v)-g(u)$
do not depend on the choice of the $\kappa_i$. To see this, suppose $\kappa_1$ is replaced in \eqref{eq:7}
by some $\kappa_1'\in \kappa_1\sH$. Since $\sH$ is a normal subgroup, $\kappa_1'=\eta\kappa_1$ for some
$\eta\in\sH$. The new function $g'$ satisfies
$$
g'(v)-g(v)=h(\kappa_1'v)-h(\kappa_1v)= h(\eta\kappa_1v)-h(\kappa_1v),
$$
so that
\begin{align*}
[g'(v)-g'(u)]-[g(v)-g(u)] = [h(\eta\kappa_1v)-h(\kappa_1v)] -[h(\eta\kappa_1u)-h(\kappa_1u)] =0,
\end{align*}
since $\eta\in\sH$ and $h$ is \hdi.

We show as follows that $g$ is $\Ga$-difference-invariant.
Let $\a\in\Ga$, and write $\a= \kappa_j\eta$ for some 
$j\in I$ and $\eta\in\sH$. Since $\Ga/\sH$ can be written in the form $\{\kappa_i\kappa_j\sH: i \in I\}$,
\begin{align*}
g(\a v) - g(\a u) &= \sum_{i\in I} \Bigl[h(\kappa_i\kappa_j\eta v) - h(\kappa_i\kappa_j\eta u)\Bigr]\\
&=g(\eta v) - g(\eta u)\\
&= g(v)-g(u),
\end{align*}
since $g$ is \hdi.
\end{proof}

If the function $g$ of \eqref{eq:7} is non-constant, it follows
that $(g-g(0),\Ga)$ is a transitive
\ghf, implying  by Theorem \ref{thm:1}(b) that $\mu(G) \ge \phi$. 
This is not invariably the case, as the following example indicates.

\begin{figure}
\centerline{\includegraphics[width=0.4\textwidth]{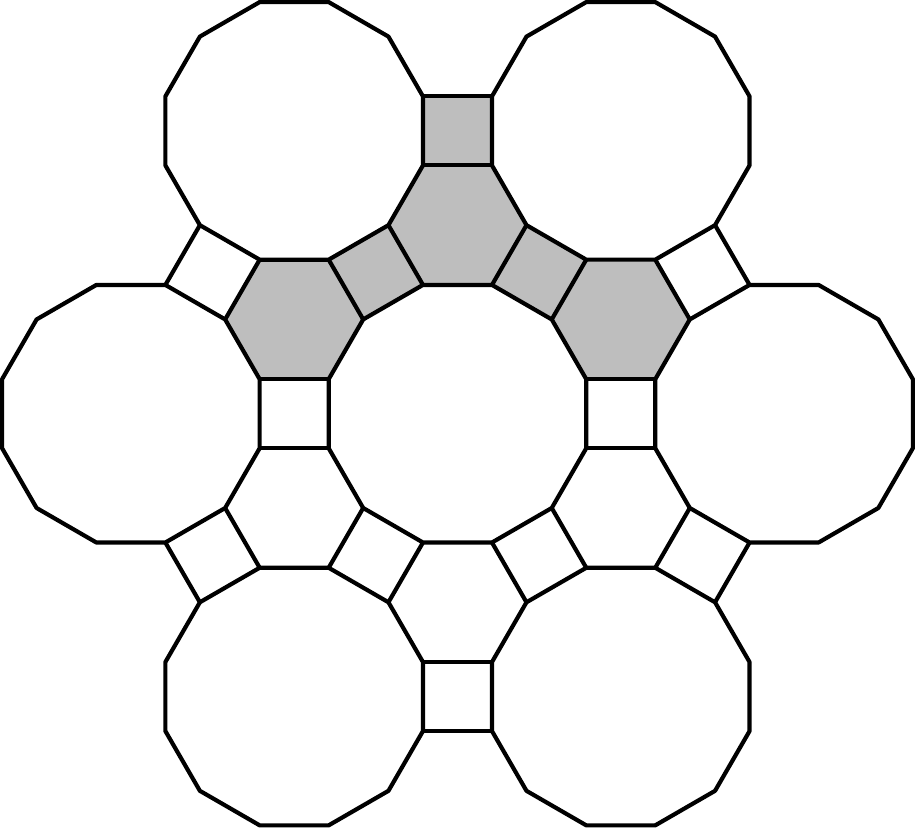}
\q\raisebox{1.8cm}{\includegraphics[width=0.5\textwidth]{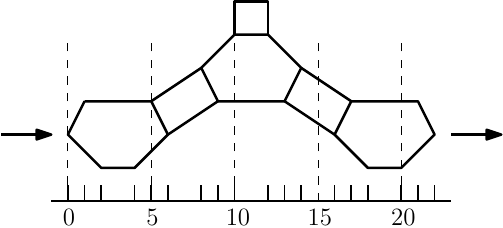}}}
\caption{The left figure depicts part of the Archimedean lattice $\AA=\arc{4,6,12}$.
Potentials may be assigned to the vertices as illustrated in the right figure,
and the potential differences are duplicated by translation, and by reflection
in a horizontal axis. 
The resulting harmonic function satisfies \eqref{eq:1}.}
\label{fig:arch}
\end{figure}

\begin{example}\label{ex:arch4}
Consider the Archimedean lattice $\AA=\arc{4,6,12}$ of Figure \ref{fig:arch}. 
Then $\AA$ is transitive and cubic,
but it has no transitive \ghf. This is seen by examining the structure of $\AA$. 
There are a variety of ways of showing $\mu(\AA)\ge\phi$,
and we refer the reader to Theorem \ref{thm:1} and the stronger inequality
of Remark \ref{rem:arch-hex2}.
\end{example}

\begin{theorem}\label{qtt}
Let $\Ga$ act transitively on $G=(V,E)\in \sG_3$.
Let $(h,\sH)$ be a \ghf\ of $G$, where $\sH \normal \Ga$ and $[\Ga:\sH]<\oo$.
Pick $\kappa_i\in \Ga$ such that $\Ga/\sH=\{\kappa_i\sH: i \in I\}$, and let
$g:V\to\ZZ$ be given by \eqref{eq:7}. If there exists a constant $C<\oo$ such that
\begin{equation}\label{eq:distbnd}
d_G(v,\kappa_i v) \le C,\qq v \in V, \ i\in I,
\end{equation}
then $(g-g(0),\Ga)$ is a transitive \ghf.
\end{theorem}

\begin{proof}
By the comment prior to Example \ref{ex:arch4}, 
we need to show that $g$ is non-constant.
Since $(h,\sH)$ is a graph height function, we may pick $v\in V$ such that
$h(v) > 2C\de$, where
$$
\de:= \max\{|h(v)-h(u)|: u \sim v\}.
$$
By \eqref{eq:distbnd}, 
$$
[h(\kappa_i v) - h(\kappa_i)]  \in [h(v)-h(\id)] + [-2C\de,2C\de].
$$
Therefore, by \eqref{eq:7},
$$
[g(v)-g(\id)]  \in |I|h(v) + \bigl[-2C\de|I|,2C\de|I|\bigr],
$$
so that $g(v) > g(\id)$ as required.
\end{proof}

\begin{corollary}
Let $\Ga=\langle S\mid R\rangle$ be an infinite, finitely-generated group. 
Let $\sH\normal \Ga$ be a finite-index normal 
subgroup, and let $(h,\sH)$ be a graph height function of the Cayley graph $G$
(so that it is a \lq strong\rq\ \ghf, see \cite{GL-Cayley}). 
Pick $\kappa_i\in \Ga$ such that $\Ga/\sH=\{\kappa_i\sH: i \in I\}$, and let
$g:V\to\ZZ$ be given by \eqref{eq:7}.
If 
\begin{equation}
\max_{1\leq i\leq k} \bigl|[\kappa_i] \bigr| < \infty, \label{mc}
\end{equation}
where $[\kappa_i]=\{g^{-1}\kappa_ig:g\in \Ga\}$ is the conjugacy class of 
$\kappa_i$,
then $(g-g(0),\Ga)$ is a transitive \ghf.
\end{corollary}

\begin{proof}
Since
$d_G(g,\kappa_ig) = d_G(0, g^{-1}\kappa_i g)$,
condition \eqref{eq:distbnd} holds by \eqref{mc}.
\end{proof}

\begin{example}
An \emph{FC-group} is a group all of whose conjugacy 
classes are finite (see, for example, \cite{tomk}). Clearly,
\eqref{mc} holds for FC-groups. 
\end{example}

We note a further situation in which there exists a transitive \ghf.

\begin{theorem}\label{thm:exact}
Let $\Ga$ act transitively on $G=(V,E)\in \sG_d$ where $d \ge 3$, and let $(h,\sH)$
be a \ghf\ on $G$. If there exists a short exact sequence 
$\id \to K \xrightarrow{\a} \Ga \xrightarrow{\be} \sH \to \id$ with
$|K|<\oo$, then $G$ has a transitive \ghf.
\end{theorem}

\begin{proof}
Suppose such an exact sequence exists. 
Fix a root $v_0\in V$, 
find $\g\in\Ga$ such that $v=\g v_0$, and define $g(v):= h(\be_\g v_0)$.

Certainly $g(v_0)=0$ and $g$ is non-constant.
It therefore suffices to show that $g$ is $\Ga$-difference-invariant.
Let $u,v\in V$ and find 
$\g'\in\Ga$ such that $ u=\g'v_0$.
For $\rho\in\Ga$,
\begin{align*}
g(\rho v)-g(\rho u) &= h(\be_{\rho\g} v_0)-h(\be_{\rho\g'} v_0)\\
&=h(\be_\rho\be_\g v_0)-h(\be_\rho\be_{\g'}v_0)\\
&=h(\be_\g v_0)-h(\be_{\g'} v_0) \qq\text{since  $\be_\rho\in \sH$}\\
&=g(v)-g(u),
\end{align*}
and the proof is complete.
\end{proof}

\section{Graphs with girth $3$ or $4$}\label{sec:girth}

As stated in Section \ref{sec:def}, $\sG_{d,g}$ denotes the subset of $\sG$ containing 
graphs with degree $d$ and girth $g$. 
Our next theorem is concerned with
$\sG_{3,3}$, and the following (Theorem \ref{g4}) with $\sG_{3,4}$.

\begin{theorem}\label{g3}
For $G\in\sG_{3,3}$, we have that
\begin{equation}\label{uc-2}
x_1\leq \mu(G)\leq x_2,
\end{equation}
where $x_1,x_2\in(1,2)$ satisfy
\begin{align}
\frac{1}{x_1^2}+\frac{1}{x_1^3}&=\frac{1}{\sqrt{2}},\label{up1}\\
\frac{1}{x_2^2}+\frac{1}{x_2^3}&=\frac{1}{2}.\label{up3}
\end{align}
Moreover, the upper bound $x_2$ is sharp.
\end{theorem}

The bounds of \eqref{up1}--\eqref{up3} satisfy 
$x_1\approx 1.529< 1.618 \approx \phi$ and $x_2 \approx 1.769$,
so that $\phi\in(x_1,x_2)$.
The upper bound $x_2$ is achieved by the Fisher graph of the 
$3$-regular tree (see Proposition \ref{lem:f4} and \cite{Gilch,GrL2}).

\begin{theorem}\label{g4}
For $G\in \sG_{3,4}$, we have that
\begin{equation}\label{uc-1}
y_1\leq \mu(G)\leq y_2,
\end{equation}
where 
\begin{align}
y_1&=12^{1/6},\label{y1}
\end{align}
and $y_2=1/\zeta$ where $\zeta$ is the smallest positive root of the equation
\begin{equation}\label{y0}
2x^2(1+x+x^2)=1.
\end{equation}
Moreover, the upper bound $y_2$ of \eqref{uc-1} is sharp.
\end{theorem}

The lower bound of \eqref{y1} satisfies 
$12^{1/6}\approx 1.513< 1.618 \approx \phi$. 
The  upper bound is approximately $y_2 \approx 1.900$,
and is achieved by the Fisher graph of the $4$-regular tree
(see Proposition \ref{lem:f4}).
The proofs of Theorems \ref{g3} and \ref{g4}
are given later in this section.

The emphasis of the current paper is upon lower bounds for connective
constants of cubic graphs. The upper bounds of Theorems \ref{g3}--\ref{g4} are included
as evidence of the accuracy of the lower bounds, and in support of the
unproven possibility that $\mu\ge\phi$ in each case. We note a more general result
(derived from results of \cite{Gilch, Wo}) for upper bounds of connective constants as follows.

\begin{theorem}\label{g5} 
For $G \in \sG_{d,g}$ where $d,g \ge 3$, we have that
$\mu(G) \le y$ where $\zeta := 1/y$ is the smallest positive real  root of the equation 
\begin{equation}\label{eq:theeq}
(d-2)\frac {M_1(\zeta)}{1+M_1(\zeta)} + \frac{M_2(\zeta)}{1+M_2(\zeta)}=1,
\end{equation}
and
\begin{equation} \label{eq:theeq2}
M_1(\zeta)=\zeta, \qq M_2(\zeta)=2(\zeta+\zeta^2+\dots + \zeta^{g-1}).
\end{equation}
The upper bound $y$ is sharp, and is achieved by the free product graph 
$F := K_2 * K_2 * \dots * K_2 * \ZZ_g$, with $d-2$ copies
of the complete graph $K_2$ on two vertices and one copy
of the cycle $\ZZ_g$ of length $g$.
\end{theorem}

See \cite{Gilch} for the definition of free product graphs. Rather than repeat the general definition here, we
explain that the extremal graph $F$ of this theorem is the (simple) Cayley
graph of the free product group $\langle S \mid R\rangle$ with
$S=\{a_1,a_2,\dots,a_{d-2},b\}$ and $R=\{a_1^2,a_2^2,\dots,a_{d-2}^2,b^g\}$.

\begin{figure}
\centerline{\includegraphics[width=0.5\textwidth]{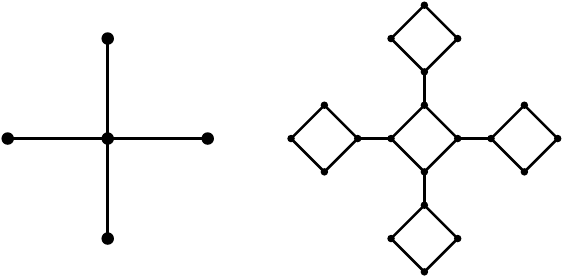}}
\caption{Each  vertex of $G$ is
replaced in the Fisher graph $G_\F$ by a cycle.}
\label{fig:fisher4}
\end{figure}

The proofs follow.
Let $G =(V,E)\in \sG_d$ where $d \ge 3$. A (generalized)  
\emph{Fisher graph} $G_\F$ is obtained from 
$G$ by replacing each vertex
by a $d$-cycle, called a \emph{Fisher cycle},
as illustrated in Figure \ref{fig:fisher4}.
The Fisher transformation originated in the work of Fisher \cite{fisher}
on the Ising model. 
We shall study the relationship between $\mu(G)$ and $\mu(G_\F)$, and to 
that end we need $G_\F$ to be quasi-transitive (see \eqref{def:cc}).
When $d=3$, $G_\F$ is invariably quasi-transitive but, when $d\ge 4$,
one needs to be specific about the choice of the Fisher cycles. Let $v \in V$, and order the
neighbours of $v$ in a fixed but arbitrary manner as $(u_1,u_2,\dots,u_d)$.
We replace $v$ by a Fisher cycle, denoted $F_v$,  
with ordered vertex-set in one--one correspondence with
the edges $\langle v, u_i\rangle$, $i=1,2,\dots,d$, in that order. For $x\in V$, find $\a_x\in\Aut(G)$
such that $\a_x(v)=x$, and replace $x$ by the Fisher cycle $\a_x(F_v)$. 
The family $\{\a_x: x \in V\}$ acts quasi-transitively on $G_\F$, as required.

The following proposition relates the connective constants of $G$ and $G_\F$, and it is valid in the
slightly more general context of non-simple graphs. Let $\sN_d$ be the set of infinite, rooted,
connected, transitive graphs with degree $d$ (we do not assume these graphs are simple).
A Fisher graph $G_\F$ of $G \in \sN_d$ is given as for simple graphs.

\begin{proposition}\label{lem:f4}
Let $G\in \sN_d$ where $d \ge 3$. 
\begin{letlist}
\item \cite[Thm 1(a)]{GrL2}  If $d=3$,
 \begin{equation}\label{eq:mudiff3}
\frac{1}{\mu(G_\F)^2}+\frac{1}{\mu(G_\F)^3} = \frac{1}{\mu(G)}.
\end{equation}
\item If $d =2r \ge 4$ is even,
 \begin{equation}\label{eq:mudiffeven}
\frac{2}{\mu(G_\F)^{r+1}}\leq \frac{1}{\mu(G)}.
\end{equation}
\item If $d=2r+1 \ge 5$ is odd,
 \begin{equation}\label{eq:mudiffodd}
\frac{1}{\mu(G_\F)^{r+1}}+\frac{1}{\mu(G_\F)^{r+2}}
\leq \frac{1}{\mu(G)}.
\end{equation}
\end{letlist}
\end{proposition}

\begin{figure}
\centerline{\includegraphics[width=0.6\textwidth]{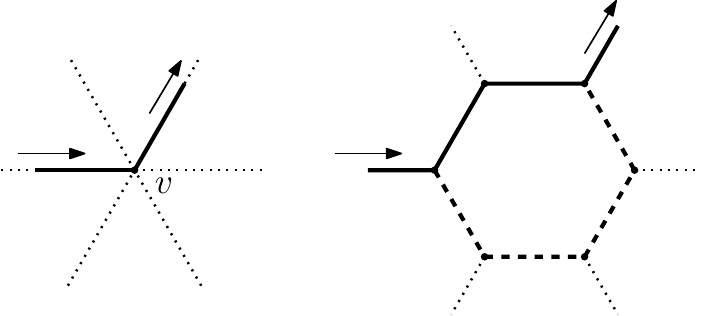}}
\caption{A degree-$6$ vertex $v$ is replaced by a Fisher $6$-cycle.
A SAW passing through $v$ may be redirected around the cycle as shown.
The entry and exit of the SAW at the Fisher cycle
traverses either $2$ edges clockwise, or $4$ edges anticlockwise.}
\label{fig:f4}
\end{figure}

\begin{proof}[Proof of Proposition \ref{lem:f4}]
We use the methods of \cite{GrL2}, where a proof of part (a) appears at Theorem 1.
(Reference \cite{GrL2} was directed at \emph{simple} graphs only, but the proof of \cite[Thm 1]{GrL2}
is valid also in the non-simple case. Indeed, there exists a unique non-simple $G \in \sN_3$.) 

Here is an outline of the proof. Consider SAWs on $G$ and $G_\F$
that start and end at midpoints of edges. Given such a SAW $\pi$ on $G$, we shall construct a corresponding SAW
$\pi'$ on $G_\F$. When $\pi$ reaches a vertex $v$ of $G$, $\pi'$ is directed
around the corresponding $d$-cycle $C$ of $G_\F$. There are $d-1$
possible exit points of $\pi'$ from $C$. For each such point, $\pi'$ may
be sent around $C$ either clockwise or anticlockwise 
(as illustrated in Figure \ref{fig:f4}). If the exit lies $s$ ($\le d/2$)
edges along $C$ from the entry, 
a single step of $\pi$ becomes a walk of
length either $s+1$ or $d-s+1$. 
Such a substitution is made at each vertex of $\pi$, resulting in a  SAW $\pi'$
on $G_\F$.

We formalize the above, thereby extending the arguments of \cite{GrL2}.
Let $d=2r\ge 4$ 
(the case of odd $d$ is similar).  
Write  $G=(V,E)$ and $G_\F=(V_\F,E_\F)$. The set $E$ may be considered as a
subset of $E_\F$.  By the argument leading to \cite[eqn (15)]{GrL2}, it suffices 
to consider SAWs on $G_\F$ that begin and end at midpoints of edges of $E$.

The generating function of SAWs beginning at a given
midpoint $e$ of $E$ on the graph $G$ is given by
\begin{equation}\label{eq:new456}
Z(\zeta) = \sum_{\pi\in \Si(G)} \zeta^{|\pi|},
\end{equation}
where $\Si(G)$ is the set of such SAWs. Let $Z_\F$ be the generating
function of SAWs on $G_\F$ starting at $e$ and ending in the set of midpoints of $E$.
The function $Z_\F$ is derived as follows. For $\pi\in\Si(G)$, let
$e_0,e_1,\dots,e_n$ be the midpoints visited by $\pi$, and let $C_i$
be the Fisher cycle of $G_\F$ touching $e_i$ and $e_{i+1}$. 
Considering the $e_i$ as midpoints of $E_\F$, let $k_i$ be the
length of the shorter of the two routes from $e_i$ to $e_{i+1}$
around $C_i$. We replace the product $\zeta^{|\pi|}$ in \eqref{eq:new456}
by 
$$
P_\pi(\zeta) := \prod_{i=0}^{n-1} (\zeta^{k_i+1}+\zeta^{d-k_i+1})
$$
to obtain
$$
Z_\F(\zeta) = \sum_{\pi\in\Si(G)} P_\pi(\zeta).
$$
Since $1\le k_i \le r$, we deduce that
\begin{equation}\label{eq:mudiff2}
Z_\F(\zeta) \ge Z\bigl(\min\{\zeta^2+\zeta^d, \zeta^3+\zeta^{d-1}, \dots, 2\zeta^{r+1}\}\bigr), \qq \zeta \ge 0.
\end{equation}
The radius of convergence of $Z_\F$ is $1/\mu(G_\F)$, and 
\eqref{eq:mudiffeven} follows from \eqref{eq:mudiff2} on letting
$\zeta\uparrow 1/\mu(G_\F)$ and noting that the minimum
in \eqref{eq:mudiff2} is achieved by $2\zeta^{r+1}$.
\end{proof}

\begin{lemma}\label{uc3}
Let $G=(V,E)\in\sG_{3,3}$. 
\begin{letlist}
\item For $v\in V$, there exists exactly one 
triangle passing through $v$. 
\item If each such triangle of $G$ is 
contracted to a single vertex,
the ensuing graph $G'$ satisfies $G'\in \sG_3$.
\end{letlist}
\end{lemma}

\begin{proof}
(a) Assume the contrary: 
each $u\in V$ lies in two or more triangles.
Since $\deg(u)=3$, there exists $v\in V$ such that 
$\langle u,v\rangle$
lies in two distinct triangles, and we write $w_1$, $w_2$ for
the vertices of these triangles other than $u$, $v$. Since each $w_i$ has degree $3$, we have than $w_1\sim w_2$. 
This implies that $G$ is finite, which is a contradiction.

(b) Let $\sT$ be the set 
of triangles in $G$, so that the elements of $\sT$ are vertex-disjoint.
We contract each $T \in \sT$ to a vertex,  thus obtaining the graph 
$G'=(V',E')$. Since each vertex of $G'$ arises from a triangle of $G$, 
the graph $G'$ is cubic, and $G$ is the Fisher graph of $G'$. 
Since $G$ is infinite, so is $G'$.

We show next that $G'$ is transitive. 
Let $v_1',v_2'\in V'$, and write $T_i=\{a_i,b_i,c_i\}$,
$i=1,2$,  for the
corresponding  triangles of $G$. Since $G$ is transitive,
there exists $\a\in\Aut(G)$ such that $\a(a_1) = a_2$.
By part (a), $\a(T_1)=T_2$. Since $\a\in\Aut(G)$,
it induces an automorphism $\a'\in\Aut(G')$ such that
$\a'(v_1')=v_2'$, as required.

Finally, we show that $G'$ is simple. If not, there exist
two vertex-disjoint triangles of $G$, $T_1$ and $T_2$ say, 
with two edges between their vertex-sets. Each vertex
in these two edges belongs to (two) faces of size $3$ and $4$.
By transitivity, every vertex has this property.
By consideration of the various possible cases, one arrives
at a contradiction. 
\end{proof}

\begin{proof}[Proof of Theorem \ref{g3}] 
Since $G$ is the Fisher graph of $G'\in \sG_3$,
by Proposition \ref{lem:f4}(a),
\begin{equation*}
\frac{1}{\mu(G)^2}+\frac{1}{\mu(G)^3}=\frac{1}{\mu(G')}.
\end{equation*}
By \cite[Thm 4.1]{GL-Comb},
\begin{equation*}
\sqrt{2}\leq \mu(G')\leq 2,
\end{equation*}
and \eqref{uc-2} follows. When $G'$ is the $3$-regular tree $T_3$, 
we have $\mu(G')=2$, and the upper bound is achieved.
\end{proof}

The following lemma is preliminary to the proof of
Theorem \ref{g4}.
 
\begin{lemma}\label{up4}
Let $G=(V,E)\in\sG_{3,4}$. 
If $G$ is not the doubly infinite ladder 
$\LL$, 
each $v \in V$ belongs to exactly one \qua. 
\end{lemma}

\begin{figure}
\centerline{\includegraphics[width=0.7\textwidth]{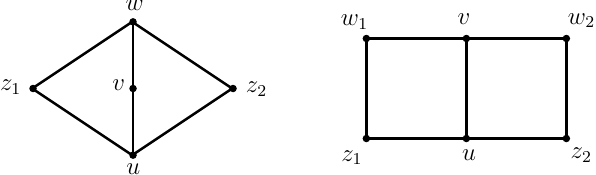}}
\caption{The two situations in the proof of Lemma \ref{up4}.}
\label{fig:2quads}
\end{figure}

\begin{proof}
Let $G=(V,E)\in \sG_{3,4}$ and $v\in V$. Assume $v$ belongs
to two or more  \qua s. We will deduce that  $G=\LL$.

By transitivity,  there exist 
two (or more) \qua s passing through every vertex $v$, and we
pick two of these, denoted $C_{v,1}$, $C_{v,2}$. 
Since $v$ has degree $3$, exactly one of the following occurs
(as illustrated in Figure \ref{fig:2quads}).
\begin{letlist}
\item $C_{v,1}$ and $C_{v,2}$ share two edges incident to $v$.
\item $C_{v,1}$ and $C_{v,2}$ share exactly one edge incident to $v$.
\end{letlist} 

\emph{Assume first that Case (a) occurs}, and let $\Pi_x$ be the property
that $x \in V$ belongs to three (or more) \qua s, any two of which share 
exactly one incident edge of $x$, these $\binom 32=3$ edges being distinct.

Let $\langle u,v\rangle$ and 
$\langle w,v\rangle$ be the two edges shared by $C_{v,1}$ 
and $C_{v,2}$, and write $C_{v,i}=(u,v,w,z_i)$,
$i=1,2$. 
Since  $u$ lies in the \qua s $C_{v,1}$, $C_{v,2}$, and $(u,z_1,w,z_2)$, we have that $\Pi_u$ occurs.
By transitivity, $\Pi_x$ occurs for every $x$.

Let $x$ be the adjacent vertex of $v$ other than $u$ and $w$. 
Note that $x\notin\{z_1,z_2\}$ and $x \not\sim u,w$, 
since otherwise $G$ would have girth $3$. By $\Pi_v$, either 
$x\sim z_1$ or $x \sim z_2$. Assume without loss of generality
that $x\sim z_1$. If $x\sim z_2$ in addition, then $G$ is finite,
which is a contradiction. Therefore, $x \not\sim z_2$.

Let $y$ be the incident vertex of 
$z_2$ other than $u$ and $w$, and note that $y\notin\{u,v,w,x,z_1,z_2\}$. By $\Pi_{z_2}$, 
there exists a \qua\ containing both $\langle y,z_2\rangle$ and 
$\langle z_2,u\rangle$. Since $u$ has degree $3$, 
either $y\sim z_1$ or $y \sim v$. However, neither is possible 
since both $z_2$ and $v$ have degree $3$.
Therefore, Case (a) does not occur. 

\emph{Assume Case (b) occurs}, and write $C_{v,i}=(u, v, w_i, z_i)$, 
$i=1,2$, for the above two \qua s passing through $v$.  Let $\Pi_x^2$
(\resp, $\Pi_x^3$) 
be the property that $x \in V$ belongs to two \qua s (\resp,
three
\qua s), and each 
incident edge of $x$ lies in at least 
one of these \qua s (\resp, every pair
of incident edges of $x$ lies in at least one of these \qua s).  
Since $\Pi_v^2$ occurs,
by transitivity $\Pi_x^2$ occurs for every $x \in V$.

Since $G$ is infinite, there exists a \lq new' edge incident to
the union of $C_{v,1}$ and $C_{v,2}$. Without loss of generality,
we take this as $\langle z_1,x\rangle$ with 
$x \notin\{u,v,w_1,w_2,z_1,z_2\}$.
By $\Pi_{z_1}^2$, there exists a \qua\ of the form $(z_1,x,y,z)$. 
Since $G$ is simple with degree $3$ and girth $4$, and $d_G(y,z_1)=2$, 
$y \notin \{z_1,u,v,w_1,w_2\}$. 

\begin{figure}
\centerline{\includegraphics[width=0.41\textwidth]{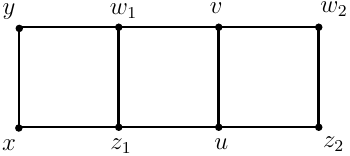}}
\caption{The second diagram of Figure \ref{fig:2quads} is extended by one further \qua.}
\label{fig:3quads}
\end{figure}

We prove next that $y \ne z_2$.
If $y=z_2$, then $\Pi^3_u$ occurs, whence $\Pi_{z_1}^3$ occurs 
by transitivity. Therefore, there exists a \qua\ passing through
the two edges $\langle x,z_1\rangle$, $\langle z_1,w_1\rangle$,
and we denote this $(x,z_1,w_1,y')$. It is immediate that
$y' \notin \{u,v,w_2,z_2\}$ since $G$ is simple with degree $3$
and girth $4$, and therefore $y'$ is a \lq new' vertex.
By $\Pi_{w_1}^3$, $y' \sim w_2$, and $G$ is finite, a contradiction.
Therefore, $y \ne z_2$, and hence $y$ is a \lq new' vertex, and $z=w_1$.
In summary, the two \qua s of the right side of Figure \ref{fig:2quads}
have been extended by adding a third \qua\ on the left side, as
illustrated in  Figure \ref{fig:3quads}. By inspection of 
the latter figure, we see that $\Pi^2_u$ occurs but not $\Pi^3_u$.

We now iterate the above procedure, adding at each stage a new \qua\ to
the graph already obtained. Suppose $G$ contains a finite, connected subgraph $S$ of
the ladder $\LL$ comprising $k$ ($\ge 3$) \qua s. Since $G$ is infinite, it contains some \lq new' edge
$e$ with exactly one endpoint in $S$. By the above considerations applied to $e$, we deduce
that $G$ contains a subgraph of $\LL$ comprising $k+1$ \qua s. 
We continue by induction to find that $G=\LL$. 
\end{proof}

\begin{proof}[Proof of Theorem \ref{g4}]
If $G=\LL$, then $\mu=\phi$, which satisfies \eqref{uc-1}. We may therefore assume that $G \ne \LL$.

Let $\sT$ 
be the set of \qua s of $G$. By Lemma \ref{up4}, each vertex lies in exactly one element of $\sT$. 
We contract each element of $\sT$ to a degree-$4$ vertex, 
thus obtaining a graph $G'$.  We claim that 
\begin{equation}\label{eq:claim}
\text{$G'\in\sN_4$, and $G$ is a Fisher graph of $G'$.}
\end{equation}
Suppose for the moment that \eqref{eq:claim} is proved.
By \cite[Thm 4.1(b)]{GL-Comb},  $\mu(G') \ge \sqrt 3$, and,
by Proposition \ref{lem:f4}(b),
\begin{equation*}
\frac2{\mu(G)^3} \leq \frac 1{\mu(G')} \le \frac{1}{\sqrt{3}},
\end{equation*}
which implies $\mu(G)\geq 12^{1/6}$.

We prove \eqref{eq:claim} next. It suffices that $G'
=(V',E') \in \sN_4$,
and $G$ is then automatically the required Fisher graph. Evidently,
$G'$ has degree $4$. We show next that $G'$ is transitive. 
Let $v_1',v_2'\in V'$, and write $C_i=(a_i,b_i,c_i,d_i)$,
for the unique  \qua\ of $G$ corresponding to $v_i'$. Since $G$ is transitive,
there exists $\a\in\Aut(G)$ such that $\a(a_1) = a_2$.
By Lemma \ref{up4}, $\a(C_1)=C_2$. Since $\a\in\Aut(G)$,
it induces an automorphism $\a'\in\Aut(G')$ such that
$\a'(v_1')=v_2'$, as required.  In conclusion, \eqref{eq:claim} holds.

For the upper bound, we refer to the following proof of the more general Theorem \ref{g5}.
\end{proof}

\begin{proof}[Proof of Theorem \ref{g5}]
Let $G\in \sG_{d,g}$ where $d,g\ge 3$, and 
let $F$ ($\in \sG_{d,g}$) be the given free product graph (see the statement of the theorem, and the remark
that follows it). 
By \cite[Thm 11.6]{Wo}, $F$ covers $G$.
Therefore, there is an injection from SAWs on $G$ with a given root to 
a corresponding set on $F$,  whence $\mu(G)\le\mu(F)$.

By \cite[Thm 3.3]{Gilch}, $\mu(F)=1/\zeta$ where $\zeta$ is the 
smallest positive real root of \eqref{eq:theeq}. 
\end{proof}

\section{The Grigorchuk group}\label{sec:Grig}

The Grigorchuk group $\Ga$ was introduced in \cite{Grig80} (see also the more recent
papers \cite{Grig84, RG05}) as a group of intermediate growth. It
is defined as follows.
Let $T$ be the rooted binary tree with root labelled $\es$. 
The vertex-set of $T$ can be identified with the set of finite strings $u$ having entries $0$, $1$, 
where the empty string corresponds to the root $\es$.
Let $T_u$ be the subtree of all vertices with root labelled $u$. 

Let $\Aut(T)$ be the automorphism group of $T$, and
let $a\in\Aut(T)$ be the automorphism
that, for each string $u$, interchanges the two vertices $0u$ and $1u$
together with their subtrees. 

Any $\g\in\Aut(T)$ may be applied
in a natural way to either subtree $T_i$, $i=0,1$. 
Given two elements 
$\g_0,\g_1\in\Aut(T)$, we define $\g=(\g_0,\g_1)$ to be the automorphism
of $T$ obtained by applying $\g_0$ to $T_0$ and $\g_1$ to $T_1$.
Define automorphisms $b$, $c$, $d$ of $T$ recursively as follows:
\begin{equation}\label{eq:grigrels}
b=(a,c),\quad c=(a,d),\quad d=(\id,b),
\end{equation}
where $\id$ is the identity automorphism. 
The Grigorchuk group $\Ga$ is defined as the subgroup of  $\Aut(T)$ generated by the 
set $S=\{a,b,c\}$. 
Denote by $G$ the Cayley graph of $\Ga$ endowed with
the generator set $S$. Since each element of $S$ has order $2$, 
we may label an edge of $G$
by the corresponding generator; an edge labelled $g$ is called a $g$-edge.  

\begin{figure}[htbp]
\centerline{\includegraphics*[width=0.45\hsize]{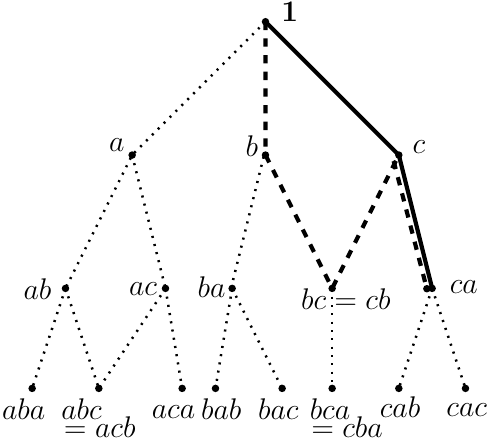}}
   \caption{The $3$-neighbourhood of the identity. The walk $(\id,c,ca)$
   may be re-routed as $(\id,b,bc,c,ca)$.}
   \label{fig:G}
\end{figure}

The $3$-neighbourhood of  $\id$ in the Cayley graph $G$ of $\Ga$
is drawn in Figure \ref{fig:G}. Since 
$G \in \sG_{3,4}$, we have by Theorem \ref{g4}  that $y_1\le \mu(G) \le y_2$ 
where the $y_i$ are given in \eqref{y1} and \eqref{y0}.
The lower bound $\mu(G) \ge y_1$ may be improved as follows.

\begin{theorem}\label{thm:Grig}
The Cayley graph $G$ of the Grigorchuk group $\Ga$ satisfies $\mu(G) \ge \phi$.
\end{theorem}

\begin{proof}
The main ideas of this proof are due to Anton Malyshev, who has kindly given 
permission for them to be included here.
A \emph{ray} of $T$ is a SAW on $T$ starting at $\es$. The collection of all infinite rays is called
the \emph{boundary} of $T$ and denoted $\pd T$. Since each $\g\in\Ga$
preserves the root $\es$, the orbit of any $v \in T$ is a subset of the generation of $T$ containing $v$.
Since $\g\in\Ga$ preserves adjacency, $\g$ maps $\pd T$ into $\pd T$.

The \emph{orbit} $\Ga \rho$ of $\rho \in \pd T$  gives rise to a graph, 
called the \emph{orbital Schreier graph} of $\rho$, and denoted here by $S(\rho)$.
The vertex-set of $S(\rho)$ is $\Ga\rho$. For $\rho_1,\rho_2\in \Ga\rho$,
$S(\rho)$ has an edge between $\rho_1$ and $\rho_2$ if and only if $\rho_2=x\rho_1$ for some $x \in \{a,b,c\}$; 
we label this edge with the generator $x$ and call it an $x$-edge. 
(Recall that $x^2=\id$ for $x\in\{a,b,c\}$.)
Such orbital Schreier graphs have been studied in \cite{Grig-S,GLN,VY} and the references therein.

Let $1^\oo$ denote the rightmost infinite ray of $T$, 
with orbital Schreier graph $\sS:=S(1^\oo)$ illustrated
in Figure \ref{fig:one-end}.
It is standard (see, for example, \cite[Thm 7.3]{Grig-S} and \cite[p.\ 29]{VY})
that, if $\rho \in \Ga 1^\oo$, $S(\rho)$ is graph-isomorphic to 
the \emph{singly infinite} graph $\sS$
(the edge-labels may depend on the choice of $\rho$). 
If $\rho \notin \Ga 1^\oo$, $S(\rho)$ is graph-isomorphic to a certain \emph{doubly infinite} chain which does not feature in this proof.

\begin{figure}[htbp]
\centerline{\includegraphics*[width=0.8\hsize]{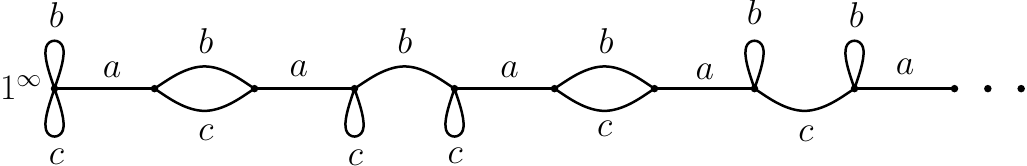}}
   \caption{The one-ended orbital Schreier graph $\sS$ of the ray $1^\oo$.}
   \label{fig:one-end}
\end{figure}

Let $\sW_0$ be the set of labelled walks on $\sS$ starting at the root $1^\oo$ that, at each step, either move one
step rightwards or pass around a loop (no loop may be traversed more than once).
Members of $\sW_0$ may be considered as words 
in the alphabet $\aleph=\{a,b,c\}$ without consecutive repetitions. 
Walks in $\sW_0$ are not generally self-avoiding on $\sS$, but we shall
see next that they give rise to a certain set $\ol{\sW_0}$
of self-avoiding walks on $G$ starting at its root $\id$.

Each  $w \in \sW_0$ lifts to a distinct walk $\ol w$ on $G$. Furthermore, we claim that
\begin{equation}\label{eq:new900}
\text{$\ol w$ is a SAW on $G$.}
\end{equation}
To see \eqref{eq:new900}, suppose $\ol w$ is
not a SAW. Then $w$ contains some shortest subword  $s$ of length $3$ or more
satisfying $s=\id$. On considering the action of $\Ga$ on $\sS$
(see Figure \ref{fig:one-end}), we deduce
that $\sS$ contains a cycle of length $3$ or more. By inspection of $\sS$, this is 
seen to be a contradiction.

Let $R=\{r_1, r_2,\dots\}$ be the set of right-hand 
endpoints of the $a$-edges of $\sS$,
labelled in the order they are encountered
when moving to the right from $1^\oo$ (in the sense of Figure \ref{fig:one-end}). An ordered pair of elements $r_i,r_j\in R$ is called
\emph{consecutive} if $|i-j|=1$. Let $\sW$ be the subset of $\sW_0$
containing words that end in $a$. 
As above, $\sW$ lifts to a set $\ol\sW$ of SAWs on $G$.
It turns out that $\sW$ is not sufficiently large to obtain $\mu(G)\ge \phi$,
and therefore we shall need to augment $\sW$ to a larger set
$\sW'$ of words, as follows. 

We think of the set $R$ as being points of renewal of walks in $\sW$.
More specifically, each $w\in \sW$ can be broken into sections  (called \emph{units})
beginning and ending (\resp) with a  consecutive pair $z, z'\in R$,
and each  unit $\si$ may be any of the following.
\begin{letlist}
\item If both $b$ and $c$ are rightward edges from $z$, 
$\si$ is a word in $\{ba,ca\}$.

\item If $b$ is rightward from $z$, and $c$ is a loop at $z$, 
 $\si$ is a word in $\{ba,cba,bca,cbca\}$. 

\item If $c$ is rightward from $z$, and $b$ is a loop at $z$, 
$\si$ is a word in $\{ca,bca,cba,bcba\}$. 
\end{letlist}
We now augment $\sW$ by replacing (a) by (a$'$), as follows.
\begin{letlist}
\item [(a$'$)] 
If both $b$ and $c$ are rightward edges from $z$, let $\si$
be any word in the set $\{ba,ca,bcba,cbca\}$. 
\end{letlist}
Let $\sW'$ be the superset of $\sW$ (viewed as sets of words
in the alphabet $\aleph$) comprising words ending in $R$, without consecutive repetitions,
and satisfying (a$'$), (b), (c). 
Let $\ol{\sW'}$ be the set of walks on $G$ obtained as lifts of elements of $\sW'$.
As above, each $w' \in \sW'$ lifts to a distinct walk $\ol {w'}\in \ol{\sW'}$.

\begin{lemma}\label{eq:new901}
Every $\ol {w'} \in \ol{\sW'}$ is a SAW on $G$. 
\end{lemma}

The proof of this lemma is given after the end of the current proof. 
The generating function $Z$ of $\ol{\sW'}$ (see \eqref{eq:genfn}) may be expressed 
in the form
$$
Z(\zeta) = A_0\sum_{n=0}^\oo A_1A_2\cdots A_n,
$$
where $A_0=2\zeta^2$ and each $A_k$, for $k \ge 1$, is either 
$$
Z_1(\zeta)=2\zeta^2+2\zeta^4 \q\text{or} \q Z_2= \zeta^2 + 2\zeta^3+ \zeta^4.
$$ 
Furthermore, $Z_1$ appears infinitely often in the sequence $(A_k:k=1,2,\dots)$.
Since $Z_1(1/\phi)>1$ and $Z_2(1/\phi)=1$, 
we have that $Z(\zeta)=\oo$ for $\zeta > 1/\phi$. The claim
of the theorem follows.
\end{proof}

\begin{proof}[Proof of Lemma \ref{eq:new901}]
For clarity of exposition, we consider first a single instance
of the \lq additional' subword $bcba$ in (a$'$), which we view
as a substitute for the unit $ca$ of (a) (the same argument applies to $cbca$ 
viewed as a substitute for $ba$). 
Let $w =x(ca)y \in \sW$ where $x$,
$y$ are words terminating with the letter $a$ (we allow $y$ to be empty), 
and let $w'=x(bcba)y$ be obtained from $w$ 
by replacing the instance of $ca$ by $bcba$. Thus, $\ol w$
is routed along the image of the $2$-path $(\id,c,ca)$ (under the action of $x$)
as indicated in Figure \ref{fig:G};
similarly,  $\ol {w'}$ is obtained from $\ol w$ by replacing this $2$-path 
by the image under $x$
of the $4$-path $(\id,b,bc=cb,c,ca)$ of the figure. The lifted walk $\ol {w'}$
fails to be a SAW only if $\ol w$ visits either 
$xb$ or $xbc$. 

There is a notational complication, arising from the need to distinguish between
elements of $\Ga$, words in the alphabet $\aleph$, and the walks on $G$ that the
last generate.

A. \emph{Suppose  $\ol w$ visits the vertex $xb$}.
\begin{romlist}
\item 
By inspection of Figure \ref{fig:one-end}, 
we have that $xb\in \sW_0$. By  \eqref{eq:new900}, $\ol x$ does not visit the vertex $xb$ of $G$
since that would contradict the fact that $\ol {xb}$ is a SAW.
Therefore, $\ol w$ visits $xb$ \emph{after} it visits the vertex $x$ of $G$.
That is, $y$ begins with a subword $y'$, with length at least $4$,
such that $x(ca)y'$ lifts to a SAW
from $\id$ to $xb$.

\item The subword $y'$ cannot end with the letter $b$ since, if it did, the penultimate vertex of $\ol{x(ca)y'}$
would be $x$, in contradiction of the fact that $\ol{x(ca)y'}$ lifts to a SAW.  

\item
Suppose $y'$ ends with the letter $a$. By inspection of Figure \ref{fig:one-end}, 
$x(ca)y'b\in\sW_0$, and hence  $x(ca)y'b$ lifts to a SAW $\ol {x(ca)y'b}$. However, $\ol {x(ca)y'b}$ 
contains a cycle containing the vertex $x$, a contradiction.

\item
Suppose $y'$ ends with the letter $c$. 
The penultimate letter of $y'$ is either $a$ or $b$. It cannot be $b$ since that would imply that
the SAW $\ol{x(ca)y'}$ visits the vertex $xc$ twice.

Therefore, $y'=y''ac$ for some word $y''$, and furthermore $\ol{x(ca)y''a}$ ends at $xbc$. As above, we have that
$x(ca)y''ab\in\sW_0$, and hence  $\ol{x(ca)y''ab}$ is a SAW. However, $\ol {x(ca)y''ab}$ 
contains a cycle containing the vertex $xc$, a contradiction.

\end{romlist}

B. \emph{Suppose  $\ol w$ visits the vertex $xbc$ but not the vertex $xb$}.
\begin{romlistp}
\item 
The walk $\ol w$ cannot visit the vertex $xbc$ before it visits the vertex $xc$, since any such visit to
$xbc$ must be followed by $xc$, in contradiction of the fact that $\ol w$ is a SAW.
That is, $y$ begins with a subword $y'$ such that $x(ca)y'$ lifts to a SAW
from $\id$ to $xbc$.

\item The subword $y'$ cannot end with $b$, since that would require two visits by the SAW $\ol w$ to the vertex $xc$.

\item The subword $y'$ cannot end in $a$, as in (iii) above.

\item If $y'$ ends in $c$, then the penultimate vertex of 
$\ol{x(ca)y'}$ is $xb$, which is excluded by assumption.
\end{romlistp}

We conclude that any substitution of a single unit
gives rise to a word in $\sW'$ that lifts to a SAW on $G$.

Suppose next that several units of a word $w\in \sW$ are altered
by substitutions of the form of (a$'$) above.
Such substitutions necessarily involve distinct units of $w$.
We consider these substitutions one by one, in the natural order of $w$.
If the new walk, which will be denoted $\ol{v'}$,  
is not self-avoiding, there is an earliest substitution which creates a cycle.
The above argument may be applied to that substitution to obtain a contradiction
in a manner similar to the above. 
We expand this slightly as follows.

Let $w=x(ca)y\in \sW$ (the case $w=x(ba)y\in \sW$ is handled similarly), and let
$w'=x(bcba)y$ be derived from $w$ by substituting $bcba$ for $ca$. 
Suppose further that certain substitutions have already 
been made to some of the units of the initial word $x\in \sW$, 
resulting in a new word $x'$. Write $v=x'(ca)y$
and $v'=x'(bcba)y$, noting that 
\begin{equation}\label{contra2}
\begin{aligned}
\text{the walks}&\ \text{$\ol w$, $\ol{w'}$, $\ol{v}$, $\ol{v'}$ 
traverse the same set of $a$-edges of $G$,}\\
&\text{in the same order and in the same directions.}
\end{aligned}
\end{equation}
 Suppose 
\begin{equation}\label{contra1}
\text{$\ol{v}$ is a SAW, but $\ol{v'}$ is not}.
\end{equation}
We will obtain a contradiction, 
and the full claim of the lemma follows.
By \eqref{contra1}, $\ol v$ visits either $xb$ or $xc$.

C. \emph{Suppose  $\ol v$ visits the vertex $xb$}.
\begin{romlist}
\item 
We prove first that $\ol{x'}$ does not visit $xb$. Suppose the converse.
By \eqref{contra2}, the final letter of both $x$ and $x'$ is $a$, so that the final step of
$\ol x$ and $\ol{x'}$ is along the directed edge
$[xa,x\rangle$. In particular, $\ol{x'}$ does not
traverse the edge $\langle x,xb\rangle$ in either direction.
As in A(i) above, $\ol x$ does not visit $xb$.
By \eqref{contra2}, $\ol {x'}$ cannot traverse the edge $\langle xba,xb\rangle$
in either direction. The claim follows.

\item
Therefore, $\ol v$ visits $xb$ \emph{after} it visits the vertex $x$ of $G$.
That is, $y$ begins with a subword $y'$ such that both $x(ca)y'$ and 
$x'(ca)y'$ lift to SAWs from $\id$ to $xb$.
This leads to a contradiction as in A(ii)--(iv) above.

\end{romlist}
D. \emph{Suppose  $\ol v$ visits the vertex $xbc$ but not the vertex $xb$}.
\begin{romlistp}
\item 
The walk $\ol v$ cannot visit the vertex $xbc$ before it visits the vertex $xc$, since any such visit to
$xbc$ must be followed immediately 
by $xc$, in contradiction of the fact that $\ol v$ is a SAW.

\item Therefore, $y$ begins with a subword $y'$ such that $x'(ca)y'$ lifts to a SAW
from $\id$ to $xbc$. This leads to a contradiction as in B(ii$'$)--(iv$'$) above.

\end{romlistp}

The proof is concluded.
\end{proof}

\section{Transitive TLF-planar graphs}\label{sec:tlf}

\subsection{Background and main theorem}
We consider next the class of so-called \lq topologically locally finite, planar graphs' (otherwise
known as TLF-planar graphs),
as defined in the next paragraph.
The basic properties of such graphs were presented in \cite{DR09}, to which the
reader is referred for further information.
In particular, the class of TLF-planar graphs includes  the one-ended
planar Cayley graphs and the transitive tilings (including the square, triangular, and hexagonal lattices).

We use the word \emph{plane} to mean a simply connected Riemann surface 
without boundaries. An \emph{embedding} of a graph $G=(V,E)$ in a plane $\sP$ is a function
$\eta:V\cup E \to \sP$ such that $\eta$ restricted to $V$
is an injection and, for $e=\langle u,v\rangle\in E$, $\eta(e)$ is a $C^1$ image of $[0,1]$. 
An embedding is ($\sP$-)\emph{planar} if the images of
distinct edges are disjoint except possibly at their endpoints, 
and a graph is ($\sP$-)\emph{planar} if
it possesses a ($\sP$-)planar embedding.
An embedding is  \emph{topologically locally finite} (\emph{TLF}) if the images of
the vertices have no accumulation point, and a
connected graph is called \emph{TLF-planar} if it possesses a planar TLF embedding. 
Let $\TLF_d$ denote the class
of transitive,  TLF-planar graphs with vertex-degree $d$.  We shall
sometimes confuse a TLF-planar graph with its TLF embedding.
The \emph{boundary} $\pd S$ of $S \subseteq \sP$ is defined by
$\partial S:= \ol S \cap \ol{\sP\setminus S}$, where $\ol T$ is the closure of a
subset $T$ of $\sP$. 

The  principal theorem of this Section \ref{sec:tlf} is as follows.

\begin{theorem}\label{tm93}
Let $G\in\TLF_3$ be infinite. Then $\mu(G)\geq \phi$.
\end{theorem}

The principal methods of the proof are as follows: (i) the construction of
an injection from eastward SAWs on $\LLp$ to SAWs on $G$,
(ii) a method for verifying that certain paths on $G$ are indeed SAWs,
and (iii) the generalized Fisher transformation of \cite{GrL2} and Section \ref{sec:girth}. 

A \emph{face} of a TLF-planar graph (or, more accurately, of its embedding)
is an arc-connected component of the (topological) complement of the graph.
The \emph{size} $k(F)$ of a face $F$ is the number of vertices in its topological boundary,
if bounded; an unbounded face has size $\oo$.
Let $G=(V,E)\in \TLF_d$ and  $v\in V$.
The \emph{type-vector} $\arc{k_1,k_2,\dots,k_d}$ of 
$v$ is the sequence of sizes of  the 
$d$ faces incident to $v$, taken in cyclic order around $v$. 
Since $G$ is transitive, the type-vector is independent 
of choice of $v$ modulo permutation of its elements, and furthermore each
entry satisfies $k_i \ge 3$. 
We may therefore refer to the type-vector $\arc{k_1,k_2,\dots,k_d}$
of $G$, and we define
\begin{equation*}
f(G)=\sum_{i=1}^d \left(1-\frac2{k_i}\right),
\end{equation*}
with the convention that $1/\oo = 0$.
We shall use the following two propositions.

\begin{proposition}[{\cite[p.\ 2827]{DR09}}]\label{th91}
Let $G=(V,E)\in\TLF_3$.
\begin{letlist}
\item If $f(G)<2$, $G$ is finite and has a planar TLF embedding in the sphere.
\item If $f(G)=2$, $G$ is infinite and has a planar TLF embedding in the Euclidean plane.
\item If $f(G)>2$, $G$ is infinite and has a planar TLF embedding 
in the hyperbolic plane (the Poincar\'e disk).
\end{letlist}
Moreover, all faces of the above embeddings are regular polygons.
\end{proposition}

There is a moderately extensive literature concerning the function $f$ and the Gauss--Bonnet 
theorem for graphs. See, for example, \cite{chen^2,HigY,Keller}.
 
\begin{proposition}\label{th92}
The type-vector of an infinite graph $G \in \TLF_3$ 
is one of the following:
\begin{Alist}
\item $\arc{m,m,m}$ with $m \ge 6$,
\item $\arc{m,2n,2n}$ with $m \ge 3$ odd, and 
$m^{-1}+n^{-1}\le\frac12$,
\item $\arc{2m,2n,2p}$ with $m,n,p\ge 2$ and
$m^{-1}+n^{-1}+p^{-1}\le 1$. 
\end{Alist} 
\end{proposition}

Recall that the elements of  a type-vector lie in 
$\{3,4,\dots\}\cup\{\oo\}$.

\begin{proof}
See \cite[p.\ 2828]{DR09} for an identification of the type-vectors
in $\sT_3$. The inequalities on $m,n,p$  arise from the condition 
$f(G) \ge 2$.
\end{proof}

\subsection{Overview and preliminary results}

Let $G \in \TLF_3$  be infinite.
By Proposition \ref{th91}, $f(G) \ge 2$. 
If $f(G)=2$ then, by Proposition
\ref{th92}, the possible type-vectors are precisely those 
with type-vectors $\arc{6,6,6}$, $\arc{3,12,12}$, $\arc{4,8,8}$, $\arc{4,6,12}$, 
$\arc{4,4,\oo}$,
that is, the hexagonal lattice \cite{ds}
and its Fisher graph \cite[Thm 1]{GrL2}, 
the square/octagon lattice \cite[Example 4.2]{GrL3}, 
the Archimedean lattice $\arc{4,6,12}$
of Example \ref{examples}(c), Example \ref{ex:arch4}, and Remark \ref{rem:arch-hex2},
and the doubly infinite ladder  of Figure \ref{fig:2}.
It is explained in the above references that each of 
these has $\mu\ge\phi$.

It suffices, therefore, to prove Theorem \ref{tm93} when
 $G \in \TLF_3$ is infinite with $f(G) > 2$.
By Proposition \ref{th92}, 
the cases to be considered are:
\begin{Alist}
\item $\arc{m,m,m}$ where $m > 6$,
\item $\arc{m,2n,2n}$ where $m\ge 3$ is odd and $m^{-1}+n^{-1}<\frac12$,
\item $\arc{2m,2n,2p}$ where $m,n,p\ge 2$ and $m^{-1}+n^{-1}+p^{-1}<1$.
\end{Alist}
These cases are covered in the following order, as indexed by section number.
\begin{enumerate}
\item [\S\ref{caseA}.] $\min\{k_i\}\ge 5$, $\arc{k_1,k_2,k_3}\ne \arc{5,8,8}$,
\item [\S\ref{caseA2}.] $\min\{k_i\}=3$,
\item [\S\ref{caseB}.] $\arc{4,2n,2p}$ where $p \ge n \ge 4$ and $n^{-1}+p^{-1}<\frac12$,
\item [\S\ref{caseC}.] $\arc{4,6,2p}$ where $p \ge 6$,
\item [\S\ref{caseD}.] $\arc{5,8,8}$.
\end{enumerate}
Note that Section \ref{caseC} includes the case of the
Archimedean lattice 
$\AA=\arc{4,6,12}$ with $f(\AA)=2$ (see also 
Example \ref{examples}(c)).
  
Let the graph $G$ lie in one of the last five categories. 
We identify $G$ with a specific planar, TLF embedding in the hyperbolic plane every face 
of which is a regular polygon. The required proof in each case is similar in overall approach to that of
Theorem \ref{thm:1}. 
Let $\WW_n$ be the set of eastward $n$-step SAWs from $0$ 
on the singly-infinite ladder  $\LLp$
of Figure \ref{fig:2}. Fix a root $v \in V$,
and let $\Si_n(v)$ be the set of $n$-step SAWs on $G$ starting at $v$. 
We shall construct an injection from $\WW_n$ to $\Si_n(v)$, and the inequality $\mu(G) \ge \phi$ will follow
by \eqref{eq:3}.

Let $w\in \WW_n$. For each of the five categories above, we shall explain how the word $w$ encodes
an element of $\Si_n(v)$.
In building an element of $\Si_n(v)$ sequentially, at each stage there is a 
choice between two new edges, which, in the sense of the embedding,
we may call `right' and `left' (when viewed from the previous edge).
The key step is to show that the ensuing paths on $G$ are indeed SAWs
so long as the cumulative differences between the aggregate numbers of right and
left steps remain sufficiently small.

Some preliminary lemmas follow. 
Let $G \in \sT_d$ be infinite,
where $d \ge 3$.
A cycle $C$ of $G$ is called \emph{clockwise} if its orientation
after embedding is clockwise. Suppose a walker traverses $C$ clockwise.
On arriving at a vertex $w$ of $C$, the walker faces $d-1$ possible exits from $w$,
the rightmost of which is designated \lq right' and the leftmost \lq left'
(the other $d-3$ are neither right nor left).
Let $r=r(C)$  (\resp, $l=l(C)$) 
be the number of right (\resp, left)  turns taken by the walker as it traverses $C$ clockwise, and let
\begin{equation}\label{eq:mudef2}
\rho(C)=r(C)-l(C).
\end{equation}

\begin{lemma}\label{lem:st}
Let $G \in \sT_d$ be infinite with $d \ge 3$.
Let $C$ be a cycle of $G$, and let 
$\sF:=\{F_1,F_2,\dots,F_s\}$ be the set of faces enclosed by $C$.
There exists $F \in \sF$ such that the 
boundary of $\sF\setminus F$ is a cycle of $G$.
The set of edges lying in $\partial F \setminus C$ forms a path. 
\end{lemma}

\begin{proof}
Let $C$ be a cycle of $G$, and let
$\sF'\subseteq\sF$ be the subset of faces that lie in the bounded
component of $\sP \setminus C$, and that share an edge with $C$.
Let $I$ be the (connected) subgraph of $G$ comprising 
the edges and vertices of the faces in $\sF'$, and
let $I_\rd$ be its dual graph (with the infinite face omitted).
Then $I_\rd$ is finite and connected, and  thus has some 
spanning tree $T$ which is non-empty. Pick a vertex $t$ of $T$ with
degree $1$, and let $F$ be the corresponding face.
The first claim follows since the removal of $t$ from $T$ 
results in a connected subtree.   The second claim holds since,
if not, the interior of $C$ is disconnected, 
which is a contradiction. 
\end{proof}

\begin{lemma}\label{l95}
Let $G \in \sT_d$ be infinite with $d \ge 3$.
For any cycle  $C=(c_0,c_1,\dots,c_n)$  of $G$,
\begin{equation}\label{eq:rl}
\rho(C)\begin{cases} =6+\displaystyle\sum_{i=1}^s \bigr[k(F_i)-6\bigr] &\text{if } d=3,\\
\ge 4+\displaystyle\sum_{i=1}^s \bigr[k(F_i)-4\bigr] &\text{if } d\ge 4,
\end{cases}
\end{equation}
where $\sF=\{F_1,F_2,\dots,F_s\}$ is the set of faces enclosed by $C$.
\end{lemma}

\begin{proof}
The proof is by induction on the number $s=s(C)$ of faces
enclosed by $C$.
It is trivial when $s=1$ that  $r(C)=k(F_1)$ and $l(C) = 0$, and 
\eqref{eq:rl} follows in that case.

Let $S \ge 2$ and assume that \eqref{eq:rl} holds for all $C$ with $s(C) < S$. 
Let $C=(c_0,c_1,\dots,c_n)$ be such that 
$s(C)=S$, and pick $F\in\sF$ as in Lemma \ref{lem:st}. Let $\pi$
be the path of edges in $\pd F \setminus C$, as illustrated in Figure \ref{fig:rightleft2}.

\begin{figure}[htbp]
\centerline{\includegraphics*[width=0.4\hsize]{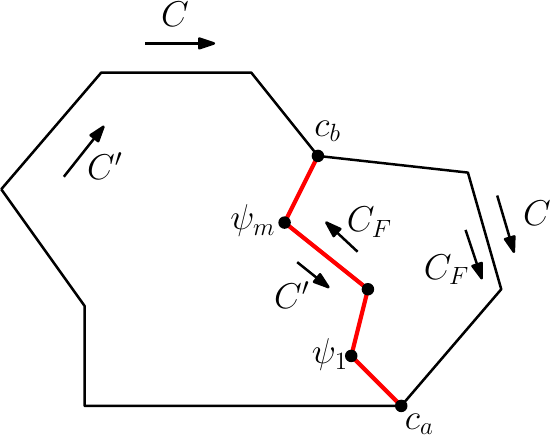}}
   \caption{An illustration of the cycles $C$, $C'$, $C_F$, with the path $\pi$ from $c_a$ to $c_b$.}
   \label{fig:rightleft2}
\end{figure}

Let $C_F$ (\resp, $C'$)  be the boundary cycle of $F$
(\resp, $\sF \setminus F$), each viewed clockwise.
We write $\pi$ in the form 
$\pi=(c_a, \psi_1,\psi_2,\dots, \psi_m, c_b)$ where $a \ne b$, $\psi_i\notin C$. 
We claim that
\begin{equation}\label{eq:rl2}
\rho(C) \begin{cases}
=  \rho(C')  + \rho(C_F)-6 &\text{if } d=3,\\
\ge\rho(C')  + \rho(C_F)-4 &\text{if } d \ge 4.
\end{cases}
\end{equation}
The induction step follows from \eqref{eq:rl2} by applying the induction hypothesis
to $C'$ and noting that $\rho(C_F)=k(F)$.

We prove \eqref{eq:rl2} by considering the contributions
made to its left and right sides by vertices
in the cycles $C , C', C_F$. Any  vertex 
$y \in C\setminus\{c_a,c_b\}$ contributes equal amounts to the left and right sides.
We turn, therefore, to vertices in the remaining path $\pi$.
\begin{numlist}
\item
The cycle $C_F$ (\resp, $C'$) takes a right  (\resp, left)
turn at each vertex $\psi_i$. The net contribution 
from $\psi_i$ to the right (\resp, left) side of \eqref{eq:rl2} is 
$1-1=0$ (\resp, $0$).
\item
Consider the turns made by $C,C',C_F$  at a vertex $x\in\{c_a,c_b\}$.
\begin{letlist}
\item Suppose $d=3$. At $x$, $C_F$ takes a  right turn, $C'$ takes a right turn, and $C$ takes a left turn. The net contribution 
from $x$ to the right (\resp, left) side of \eqref{eq:rl2} is 
$1+1=2$ (\resp, $-1$).
\item Suppose $d \ge 4$. 
At $x$,  $C_F$ takes a  right turn.
Furthermore, if $C'$ takes a right turn,
then $C$ does not take a left turn. 
The net contribution from $x$ to $[ \rho(C')+ \rho(C_F)]-\rho(C)$
is at most $2$.
\end{letlist}
\end{numlist}
We sum the above contributions, noting that case 2 applies for exactly
two values of $x$, to obtain \eqref{eq:rl2}. The proof is complete.
\end{proof}

\begin{lemma}\label{l96}
Let $G \in \sT_d$ be infinite with type-vector $\arc{k_1,k_2,\dots,k_d}$,
and let $C$ be a cycle of $G$.
\begin{letlist}
\item
If $d=3$ and  $\min\{k_i\}\geq 6$, then
$\rho(C)\geq 6$.
\item If $d=3$ and $\arc{k_1,k_2,k_3}=\arc{5,2n,2n}$ with $n\geq 5$,
then $\rho(C) \ge 5$.
\item If $d \ge 4$ and $\min\{k_i\} \ge 4$, then $\rho(C)\ge 4$.
\end{letlist}
\end{lemma}

\begin{proof}
(a,\,c) These are immediate consequences of \eqref{eq:rl}.

\par\noindent
(b) Suppose $\arc{k_1,k_2,k_3}=\arc{5,2n,2n}$ with $n\geq 5$, 
and let $M=M(C)$ be the number of size-$2n$ faces enclosed
by a cycle $C$. 
We shall prove $\rho(C) \ge 5$ 
by induction on $M(C)$.
If $M=0$, then $C$ encloses exactly one size-$5$ face, and 
$\rho(C)=5$.
Let $S \ge 1$, and assume $\rho(C) \ge 5$ for any cycle $C$ with $M(C) <S$.

Let $C$ be a cycle with $M(C)=S$. Since
every vertex of $C$ is incident  to no more than one size-$5$ face
inside $C$, $C$ contains some size-$2n$ face $F$ with at least
one edge in common with $C$. 
Let $C'$ be the boundary of the set obtained by
removing $F$ from the inside of $C$; that is, $C'$ may be viewed as the sum of
the cycles $C$ and $\partial F$ with addition modulo $2$. 
Then $C'$ may be expressed as the edge-disjoint union of
cycles $C_1,C_2,\dots,C_m$ satisfying $M(C_i)<S$ for $i=1,2,\dots,m$. 

By \eqref{eq:rl} and the induction hypothesis,
\begin{align*}
\rho(C) &=6+[2n-6]+\sum_{i=1}^m [\rho(C_i)-6]\\
&\ge 2n-m.
\end{align*}
Each $C_i$ shares an edge with $\partial F$, and no two such edges
have a common vertex. Therefore, $m \le n$, and the induction step is complete since $n \ge 5$.
\end{proof}

\subsection{Proof that $\mu\ge\phi$ when $\min\{k_i\}\ge 5$ and $\arc{k_1,k_2,k_3}
\neq \arc{5,8,8}$} 
\label{caseA}

This case covers the largest number of instances. Certain other special families of type-vectors
will be considered in Sections \ref{caseA2}--\ref{caseD}.
By Proposition \ref{th92}, it suffices to assume
\begin{equation}\label{assume1}
\text{either $\min\{k_i\}\geq 6$, 
or $\arc{k_1,k_2,k_3}=\arc{5,2n,2n}$ with $n\geq 5$}.
\end{equation}

We shall construct an injection from the set $\WW_n$ to the set $\Si_n(v)$ of SAWs on $G$ starting at $v\in V$.
For $w \in \WW_n$, we shall define an $n$-step
SAW $\pi(w)$ on $G$, and the  map $\pi:\WW_n\to \Si_n(v)$ will be an injection.
The idea is as follows.  With $G$ embedded in the plane, one may think 
of the steps of a SAW on $G$ (after its first edge) as 
taking a sequence of right and left turns.
For given $w \in \WW_n$, we will explain how the letters H and V in $w$ are mapped
to the directions right/left.  

Let $n \ge 1$ and $w =(w_1w_2\cdots w_n)\in \WW_n$, so that in particular $w_1=$ H. 
We shall construct the SAW $\pi=\pi(w)$ via an intermediate SAW $\pi'$
which is constructed iteratively as follows.
In order to fix an initial direction, we choose a $2$-step SAW $(v',v,v'')$ 
of $G$ starting at some neighbour $v'$ of $v$, and 
we assume in the following that
the turn in the path $(v',v,v'')$ is rightwards (the other case is 
similar). 
We set $\pi'(w)=(v',v,v'')$ if $n=1$. and we call this rightwards turn the \emph{first} turn of $\pi'$.
The first letter of $w$ is $w_1=$ H, and the second is either H or V,
and the latter determines whether the second turn of $\pi'$ is the same 
as or opposite to the previous turn.
We adopt the rule that:
\begin{equation}
\begin{aligned}
\text{if $(w_1w_2) =$ (HV),\q}&\text{the second turn is the same (rightwards) as the previous},\\
\text{if $(w_1w_1) =$ (HH),\q}&\text{the second turn is opposite (leftwards) to the previous}.\\
\end{aligned}\label{eq:rule1}
\end{equation}

For  $k\ge 3$, the $k$th turn of $\pi'$ is either to the right or the left, and is either the same or 
opposite to the $(k-1)$th turn. Whether it is the same or opposite is determined as follows:
\begin{equation}
\begin{aligned}
\text{when $(w_{k-2}w_{k-1}w_k) =$ (HHH),\qq}&\text{it is opposite},\\
\text{when $(w_{k-2}w_{k-1}w_k) =$ (HHV),\qq}&\text{it is the same},\\
\text{when $(w_{k-2}w_{k-1}w_k) =$ (HVH),\qq}&\text{it is opposite},\\
\text{when $(w_{k-2}w_{k-1}w_k) =$ (VHH),\qq}&\text{it is the same},\\
\text{when $(w_{k-2}w_{k-1}w_k) =$ (VHV),\qq}&\text{it is opposite}.
\end{aligned}\label{eq:rule2}
\end{equation}

When  the iterative construction is complete, a path 
$\pi'=(v'=\pi'_{-1},v=\pi'_0,v''=\pi'_1,\dots,\pi'_n)$ on $G$ ensues.
Since $\pi'$ proceeds by right or left turns,
it is  non-backtracking. The following claim will be useful in showing
it is also self-avoiding.

\begin{lemma}\label{cm}
Let $i\in\{0,1,\dots,n\}$.
For any subpath of $\pi'$ beginning at $\pi'_i$, 
the numbers of right turns and left turns differ by at most $3$. 
\end{lemma}

\begin{proof}
A subpath of $\pi'$ corresponds to some subword $w'$ of $w$, and we may
assume the length of $w'$ is at least three. Let $k\ge 1$.
A $k$-\emph{block} of $w'$ is defined to be a subword $B$ of $w'$ of the form VH$^k$V, where H$^k$
denotes $k$ consecutive appearances of H. A \emph{block} is a $k$-block for some $k \ge 1$. 
A $k$-block $B$ is called \emph{even} (\resp, \emph{odd}) according to the parity of $k$. 

A $k$-block $B$ generates $k+1$ turns in $\pi'$ corresponding to the letters H$^k$V.
These $k+1$ turns are determined by the $k+1$ triples
HVH, VHH, HHH, $\dots$, HHV (with $k-2$ appearances of HHH).  By inspection of \eqref{eq:rule2},
the corresponding turns are related to their predecessors by the sequence
osoo$\,\cdots\,$os, where o (\resp, s) means `opposite' (\resp, `same'),
that is, with $k-1$ opposites and $2$ sames. Suppose, for illustration, that the turn
immediately prior to the block was R, where R (\resp, L) denotes right (\resp, left).
Then the corresponding sequence of turns begins LLRL$\,\cdots$ 
\begin{letlist}
\item If $B$ is odd, then, in the corresponding $k+1$ turns made by $\pi'$, the
numbers of right and left turns are equal. Moreover, if the first turn is to the right 
(\resp, left), then the last turn is to the  left (\resp, right).
\item If $B$ is even, the numbers of right and left turns differ by $3$. 
Moreover, the first turn is to the left if and only if the last turn is to the left, 
and in that case there are $3$ more left turns than right turns.
\end{letlist}

Let $B$ be an odd block.
By (a), $B$ makes no contribution to the aggregate difference between the
number of right and left turns. Furthermore, the first turn of $B$ equals the first turn following $B$
(since the last turn of $B$ is opposite to the first, and the 
following subword HVH results in a turn equal to the first).
We may therefore consider $w'$ with all odd blocks removed
(which is to say, an odd $k$-block is replaced by a single V), and we assume
henceforth that $w'$ has no odd blocks.

Using a similar argument for even blocks based on (b) above, 
the effects of two even blocks 
cancel each other, and we may therefore remove any even 
number of even blocks from $w'$ (with a single V remaining) without altering the aggregate difference. 
After performing these reductions, we obtain from $w'$
a reduced word $w''$ with form
H$^a$VH$^b$, H$^a$VH$^{2r}$VH$^b$, or H$^r$, 
where $a \ge 0$, $r \ge 1$, 
$b \ge 0$. 
We consider each of these cases separately. 

Let $\De$ denote the difference (in absolute value)
between the numbers of left and right turns, and write 
$(\,)_m$ for a sequence of length $m$.
Each turn of $w''$ is related to its previous turn according to
\eqref{eq:rule2}. Therefore, once the first turn of $w''$ is determined, 
the rest follow by sequential application of \eqref{eq:rule2}. 
The first turn depends on the character (V or H) prior to $w''$,
but its value is immaterial to the value of $\De$. We may therefore
choose the previous character arbitrarily.

A. \emph{The case }$w''= $ H$^a$VH$^b$. Suppose that $w''$ is preceded by H.
\begin{romlist}
\item Let $a \ge 2$. The corresponding sequence is (ooo$\cdots$)$_{a-1}$s(osoo$\cdots$)$_b$, and $\De \le 2$.

\item Let $a=1$. The sequence is s(osoo$\cdots$)$_b$, so that $\De\le 2$.

\item Let $a=0$. The sequence  is   (osoo$\cdots$)$_b$, so that $\De \le 2$.
 
\end{romlist}

B. \emph{The case }$w''= $ H$^a$VH$^{2r}$VH$^b$.  
Suppose that $w''$ is preceded by H.
\begin{romlist}
\item Let $a \ge 2$. The sequence is (ooo$\cdots$)$_{a-1}$s(osoo$\cdots$)$_{2r}$s(osoo$\cdots$)$_b$
so that $\De \le 2$.

\item Let $a=1$. The sequence is s(osoo$\cdots$)$_{2r}$s(osoo$\cdots$)$_b$, so that $\De\le 2$.

\item Let $a=0$. The sequence  is   (osoo$\cdots$)$_{2r}$s(osoo$\cdots$)$_b$, so that $\De \le 3$.
 
\end{romlist}

C. \emph{The case }$w''= $ H$^r$.  We may suppose that $r\ge 4$,
since otherwise $\De\le 3$ trivially. 
Suppose that $w''$ is preceded by H.
The sequence is (ooo$\cdots$)$_{r-1}$,
so that $\De \le 1$.

In every case $\De \le 3$, and the proof is complete.
\end{proof}

Write $\pi'(w)=(v', v=x_0,v''=x_1,\dots,x_n)$,
and remove the first step to obtain a 
SAW $\pi(w)=(v=x_0,x_1,\dots,x_n)$.
By Lemmas \ref{l96}(a,\,b) and \ref{cm}, 
subject to \eqref{assume1}, 
$\pi(w)$ contains no cycle and is thus a SAW.
This is seen as follows.
Suppose $\nu=(x_i, x_{i+1}, \dots, x_j=x_i)$
is a cycle. The cycle has one more turn than the path, and hence,
by Lemma \ref{cm}, $|\rho(\nu)| \le 4$, in contradiction of
Lemma \ref{l96}(a,\,b).
In conclusion,  $\pi$ maps $\WW_n$ to $\Si_n(v)$. 

The map $\pi: \WW_n \to \Si_n(v)$ is an injection since, by examination
of \eqref{eq:rule1}--\eqref{eq:rule2}, $\pi(w)\ne \pi(w')$ if $w \ne w'$.
We deduce  by \eqref{eq:3} that $\mu(G)\ge\phi$.


\subsection{Proof that $\mu\ge\phi$ when $\min\{k_i\} = 3$}\label{caseA2}
Assume $\min\{k_i\}=3$.
By Proposition \ref{th92} and the assumption $f(G)>2$, 
the type-vector is $\arc{3,2n,2n}$ for some $n\geq 7$. 
On contracting each triangle to a single vertex, we obtain
the graph $G'=\arc{n,n,n}$; therefore,  $G$ is a Fisher graph of  
$G'$. 
By Proposition \ref{lem:f4}(a),
\begin{equation*}
\frac{1}{\mu(G)^2}+\frac{1}{\mu(G)^3}=\frac{1}{\mu(G')}.
\end{equation*}
It is proved in Section \ref{caseA} that $\mu(G')\ge \phi$, and the inequality $\mu(G)\ge\phi$
follows (see also \cite{GrL2}).
 
\subsection{Proof that $\mu\ge\phi$ for $\arc{4,2n,2p}$ with $p\geq n\geq 4$
 and $n^{-1}+p^{-1}<\frac12$} 
\label{caseB}
 
Let $G=(V,E)\in\TLF_3$ be infinite with type-vector $\arc{4,2n,2p}$ where $p\geq n\geq 4$ and 
$n^{-1}+p^{-1}<\frac{1}{2}$.
Note that $G$ has girth $4$, and every vertex is incident to exactly one size-$4$ face. 

From $G$, we obtain a new graph $G'$  by contracting 
each size-$4$ face to a vertex. 
For $u,v\in V$ lying in different size-$4$ faces $F_u$, $F_v$ of $G$, 
there exists $\a\in\Aut(G)$ that  maps $v$ to $v$, and hence maps $F_u$ to $F_v$.
Therefore, $\a$ induces an automorphism of $G'$, so that $G'$ is transitive.
We deduce  that $G'\in\TLF_4$, and in addition $G'$  is infinite with girth $n\ge 4$ and 
type-vector $\arc{n,p,n,p}$.
We shall make use of $G'$ later in this proof.

Let $v \in V$. We will construct an injection from $\WW_n$
to $\Si_n(v)$ in a manner similar to the argument 
following \eqref{assume1}.  
An edge of $G$ is called \emph{square} if it lies in a size-$4$ face,
and \emph{non-square} otherwise.  
Let $w =(w_1w_2\cdots w_n) \in \WW_n$.   We shall construct a 
\emph{non-backtracking} $n$-step 
walk $\pi=\pi(w)$ on $G$ from $v$, and then show it is a SAW.
For $k=1$, set $\pi(w)=(v,v')$ where $\langle v,v'\rangle$ is the unique non-square
edge of $G$ incident to $v$.
We perform the following construction for $k = 2,3,\dots, n$, in which the edges
of $\pi$ are denoted $e_1,e_2,\dots,e_n$ in order.

\begin{numlist}
\item Suppose $(w_{k-1}w_k)=$ (HV). 
The edge $e_k$ is chosen to be square according to the following rules.
   
\begin{letlist} 
\item If the edge $e_{k-1}$ of $\pi$ corresponding to $w_{k-1}$ is
square, then the next edge $e_k$ of $\pi$ is square. That is,  $e_{k-1}$
and $e_k$ form a length-$2$ path on the same size-$4$ face of $G$.
  
\item Suppose $e_{k-1}$ is non-square. Then the next edge $e_k$
is one of the two possible square edges, chosen as follows.
In contracting $G$ to $G'$,
the walk $(\pi_0,\pi_1,\dots,\pi_{k-1})$ contracts to a walk 
$\pi'=\pi'(w)$ on $G'$. 
Find the most recent turn at which $\pi'$ turns either right or left. 
If, at that turn, $\pi'$ turns left (\resp, right), 
the walk $\pi$ on $G$ turns left (\resp, right).
If no turn of $\pi'$ is rightwards or leftwards, then (for definiteness) $\pi$ turns left.
\end{letlist}
  
\item Suppose $(w_{k-1}w_k)=$ (HH).
  
\begin{letlist}
\item If the edge $e_{k-1}$ of $\pi$ corresponding to $w_{k-1}$ is
square, then the next edge $e_k$ of $\pi$ is the unique possible  non-square edge. 
  
\item Suppose $e_{k-1}$ is non-square.  
Then $e_k$
is one of the two possible square edges, chosen as follows.
In the notation of 1(b) above,
 find the most recent turn at which $\pi'$ turns either right
 or left. If at that turn, $\pi'$ turns left (\resp, right), 
 the walk $\pi$ on $G$ turns right (\resp, left).
 If $\pi'$ has no such turn, then $\pi$ turns right.
\end{letlist}
  
\item Suppose $(w_{k-1}w_k)=$ (VH), so that, in particular, $k \ge 3$. 
The edge $e_{k-1}$ of 
$G$ corresponding to $w_{k-1}$ must be square. 
If $e_{k-2}$ is square 
(\resp, non-square), then $e_k$ is the unique possible non-square (\resp, square) edge. \end{numlist}

  \begin{figure}
\centerline{\includegraphics[width=0.6\textwidth]{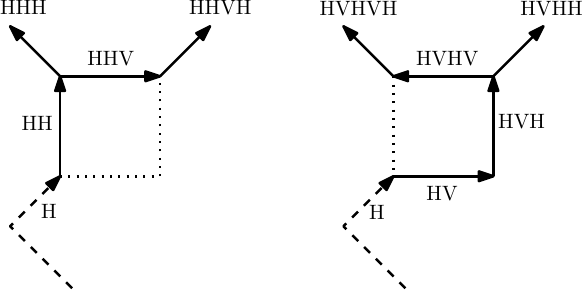}}
\caption{The dashed line is the projected SAW on $G'$. After
a right (\resp, left) turn, the projection either moves straight or turns left (\resp, right).}  
\label{hv1}
\end{figure}

We claim that the  mapping 
$\pi:\WW_n \to \Si_n(v)$ is an injection. 
By construction, $\pi(w)=\pi(w')$ if and only if $w=w'$,
and, furthermore, $\pi(w)$ is non-backtracking. 
It remains to show that each $\pi(w)$ is a SAW. 
In showing this as follows, we shall make use of the projected walk $\pi'(w)$ on the graph $G'$.

We begin the proof that $\pi=\pi(w)$ is a SAW
with some observations concerning the above construction, illustrated in part by Figure \ref{hv1}. 
\begin{romlist}
\item Every non-square edge of $\pi$ corresponds to the letter H. Thus, $\pi'=\pi'(w)$ takes a step
only (but not invariably) when H appears.
\item Each non-square edge of $\pi$ is followed by a square edge of some size-$4$ face $F$. 
Having touched a size-$4$ face 
$F$, the walk $\pi$ proceeds around $F$ before departing along 
the unique non-square edge incident with the point of departure.
\item The walk $\pi$ never traverses consecutively  more then three 
edges of any $F$. In addition, $\pi'$ is
non-backtracking.
\item The projected walk $\pi'$ takes steps on $G'$. The steps of $\pi'$
can be rightwards, straight on, or leftwards. 
If we pay no attention to the straight-on steps, then each left step is followed immediately  by a right
step, and \emph{vice versa}.
\end{romlist}

By the above, the numbers of right and left turns of $\pi'$ have 
difference at most $1$, and moreover the same holds for any subwalk $\nu$ of $\pi'(w)$.
By Lemma \ref{l96}(c) or directly, 
no such $\nu$ can form a cycle. Hence $\pi'(w)$ (and therefore $\pi(w)$ also) is a SAW.
The proof is complete.  

\begin{figure}
\centerline{\includegraphics[width=0.35\textwidth]{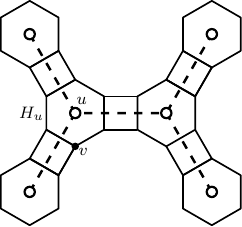}}
\caption{The graph $G$ with an embedded (dashed)
copy of the graph $P = \arc{p,p,p}$.}  
\label{111}
\end{figure}
  
\subsection{Proof that $\mu\ge\phi$ for $\arc{4,6,2p}$ with $p\geq 6$}
\label{caseC}

Let $G\in\TLF_3$ be infinite with type-vector $\arc{4,6,2p}$ 
where $p\geq 6$. 
(We include the case $p=6$, 
being the Archimedean lattice of Figure \ref{fig:arch}.)
Associated with $G$ is the graph $P:=\arc{p,p,p}$
as drawn in  Figure \ref{111}. 
As illustrated in the figure,
each vertex $u$ of $P$
lies in the interior of some size-$6$ face of $G$ denoted $H_u$. Let
$u$ be a vertex of $P$ and let $v$ be a vertex of $H_u$.
Let $\pi=(u_0=u,u_1,\dots,u_n)$ 
be a SAW on $P$ from $u$. We shall explain
how to  associate with $\pi$ a family of SAWs on $G$ from $v$.
The argument is similar to that of the proof of 
Proposition \ref{lem:f4}.

A hexagon $H$ of $G$ has six edges, which we denote  
according  to approximate compass bearing. For example, $p_\w(H)$ 
is the edge on  the west side of $H$, and similarly
$p_{\nw}$, $p_{\nea}$, $p_\e$, $p_{\se}$, $p_{\sw}$.
For definiteness, we assume that $H_u$ has
orientation as in Figure \ref{111}, and $v\in p_{\sw}(H_u)$, as in Figure \ref{112}.

\begin{figure}
\centerline{\raisebox{27pt}{\includegraphics[width=0.32\textwidth]{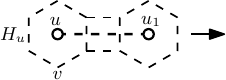}}
\q\includegraphics[width=0.5\textwidth]{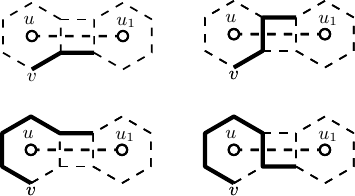}}
\caption{The step $(u,u_1)$ on $P$ may be mapped to
any of the four SAWs on $G$ from $v$, as drawn on the right.}  
\label{112}
\end{figure}

Let $\Si_n(u)$ be the set of $n$-step
SAWs on $P$ from $u$, the first edge of which is 
either north-westwards or eastwards 
(that is, away from $p_{\sw}(H_u)$).  \emph{Suppose the 
first step of the SAW $\pi\in \Si_n(u)$ is to the neighbour $u_1$
that lies eastward of $u$} (the other  cases are similar). 
With the step $(u, u_1)$, we may associate any of four SAWs on $G$
from $v$ to $p_{\w}(H_{u_1})$, namely those illustrated in Figure \ref{112}.
These paths have lengths $2$, $3$, $5$, $6$.
If $u_1$ lies to the north-west of $u$, the corresponding four paths 
have lengths $3$, $4$, $4$, $5$.

We now iterate the above construction. At each step of $\pi$,
we construct a family of $4$ SAWs on $G$ that extend the walk 
on $G$ to a new hexagon. When this process
is complete, the ensuing paths on $G$ are all SAWs, and they are distinct.

Let $Z_P(\zeta)$ (\resp, $Z_G(\zeta)$) be
the generating function of SAWs on $P$ from $u$  (\resp, on $G$
from $v$), subject to above italicized assumption.
In the above construction, each step of $\pi$ is replaced
by one of four paths, with lengths lying in either $(2,3,5,6)$
or $(3,4,4,5)$, depending on the initial vertex of the segment. 
Since
$$
\zeta^2+\zeta^3+\zeta^5+\zeta^6 \ge \zeta^3+2\zeta^4+\zeta^5
\q [= \zeta^3(1+\zeta)^2],\qq \zeta\in \RR,
$$
we have, by the argument that led to \eqref{eq:mudiff2}, that
\begin{equation}\label{eq:rad}
Z_P(\zeta^3(1+\zeta)^2) \le Z_G(\zeta), \qq \zeta \ge 0.
\end{equation}

Let $z>0$ satisfy
\begin{equation}\label{eq:gl54}
z^3(1+z)^2=\frac1{\mu(P)}.
\end{equation}
Since $1/\mu(P)$ is the radius of convergence of $Z_P$,
\eqref{eq:rad} implies $z \ge 1/\mu(G)$, which is to say that
\begin{equation}\label{eq:gl55}
\mu(G) \ge \frac 1z.
\end{equation}
As in Section \ref{caseA}, $\mu(P)\ge\phi$. It suffices for $\mu(G)\ge\phi$, therefore, to show that
the (unique) root in $(0,\oo)$ of
$$
x^3(1+x)^2=\frac1{\phi}
$$
satisfies $x \le 1/\phi$, and it is easily checked that, in fact,
$x=1/\phi$.

\begin{remark}[Archimedean lattice {$\AA=\arc{4,6,12}$}]
\label{rem:arch-hex2}
The inequality $\mu(\AA)\ge\phi$ may be strengthened.
In the special case $p=6$, we have that $\mu(P)=\sqrt{2+\sqrt 2}$;
see \cite{ds}. By \eqref{eq:gl54}--\eqref{eq:gl55},
$\mu(G) \ge 1.676$.
\end{remark}

\subsection{Proof that $\mu\ge\phi$ for $\arc{5,8,8}$}
\label{caseD}

Let $G\in\TLF_3$ be infinite with type-vector $\arc{5,8,8}$.
The proof in this case is essentially the same as that of Section \ref{caseB},
but with squares replaced by pentagons.

Let $G'$ be the simple graph obtained from $G$ by contracting each size-$5$ face of $G$ to a vertex. 
As in the corresponding step at the beginning of Section \ref{caseB},
we have that $G'\in\TLF_5$ is infinite with 
type-vector $\arc{4,4,4,4,4}$. 
A midpoint of $G$ is called \emph{pentagonal} if it belongs
to a size-5 face, and \emph{non-pentagonal} otherwise.

We opt to consider SAWs that start and end at midpoints of edges.
Let $m$ be the midpoint of some non-pentagonal edge of $G$, and let
$\Si_n(m)$ be the set of $n$-step SAWs on $G$ from $m$. 
We will find an injection from $\WW_n$ to
$\Si_n(m)$. Let $w=(w_1w_2\cdots w_n)\in \WW_n$.
We construct as follows a non-backtracking path $\pi=\pi(w)$ 
on $G$ starting from $m$. The first step of $\pi(w)$ is
$(v,v')$ where $v'$ is an arbitrarily chosen midpoint neighbouring $m$.
We write $\pi=(\pi_0,\pi_1,\dots,\pi_n)$.

For any walk $\pi'$ of $G'$, let $\rho(\pi')=r(\pi')-l(\pi')$,
where $r(\pi')$ (\resp, $l(\pi')$) is the number of right (\resp, left) turns of $\pi'$.
Since walks move between midpoints, each step of $\pi'$ involves a turn,
and thus the terminology is consistent with its previous use.

We iterate the following for $k=2,3,\dots,n$ (cf.\ the construction of Section \ref{caseB}).
\begin{numlist}
\item Suppose $(w_{k-1}w_k) = $ (HV).
The midpoint $\pi_k$ is chosen to be pentagonal according to the following rules.
\begin{letlist}

\item
 If $\pi_{k-1}$ is pentagonal, the 
next point $\pi_k$ is also pentagonal.

\item 
Suppose $\pi_{k-1}$ is non-pentagonal.
On contracting $G$ to $G'$, the 
path on $G$, so far, gives rise to a walk $\pi'$ on $G'$. 
If $\rho(\pi')<0$ (\resp, $\rho(\pi')\geq 0$), then the next turn
of $\pi$ is to the left (\resp, right).

\end{letlist}

\item Suppose $(w_{k-1}w_k) = $ (HH).
\begin{letlist}

\item If $\pi_{k-1}$ is pentagonal, then 
$\pi_k$ is non-pentagonal.

\item Suppose $\pi_{k-1}$ is non-pentagonal.
In the notation of 1(b) above, if $\rho(\pi')<0$ (\resp, $\rho(\pi')\geq 0$), then the next turn
of $\pi$ is to the right (\resp, left).
\end{letlist}

\item Suppose $(w_{k-1}w_k) = $ (VH), and note that
$\pi_{k-1}$ is necessarily pentagonal. If $\pi_{k-2}$ is pentagonal (\resp,
non-pentagonal), then $\pi_k$ is non-pentagonal (\resp, pentagonal). 
\end{numlist}

\begin{figure}
\centerline{\includegraphics[width=0.7\textwidth]{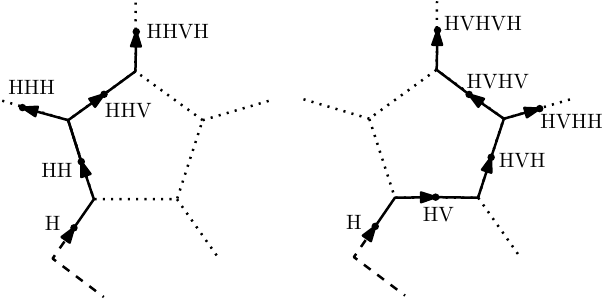}}
\caption{The dashed line is the projected SAW $\pi'$ on $G'$, assumed 
in the figure to satisfy $\rho(\pi')\ge 0$.
When $\rho(\pi')\ge 0$ (\resp, $\rho(\pi')<0$),  
the projection may move  leftwards but not rightwards 
(\resp, rightwards but not leftwards) at its next pentagon.}  
\label{hvp}
\end{figure}

We claim that the  mapping 
$\pi:\WW_n \to \Si_n(m)$ is an injection, and this claim is justified
very much as in the corresponding step of Section \ref{caseB}. 
It is straightforward that $\pi$ is an injection from $\WW_n$ to 
the set of $n$-step non-backtracking paths of $G$ from $m$,
and it suffices to show that any $\pi(w)$ is self-avoiding. 

Points (i)--(iv) of Section \ref{caseB} are replaced in the current setting by the following,
illustrated in Figure \ref{hvp}.
\begin{romlist}
\item Any step of $\pi$ leading to a non-pentagonal midpoint corresponds to the letter H. 
Thus, $\pi'=\pi'(w)$ takes a step
only (but not invariably) when H appears.
\item Each non-pentagonal midpoint of $\pi$ is followed by a pentagonal midpoint of some size-$5$ face $F$. 
Having touched a size-$5$ face 
$F$, the walk $\pi$ proceeds around $F$ before departing to the unique non-pentagonal 
midpoint available from the point of departure.
\item The walk $\pi$ never uses consecutively  more then three 
midpoints of any $F$. In addition, $\pi'$ is non-backtracking.
\item The projected walk $\pi'$ takes steps on $G'$. The steps of $\pi'$
can be rightwards,  leftwards, or \lq between'. 
If we pay no attention to the \lq between' steps, then each left step is followed immediately  by a right
step, and \emph{vice versa}.
\end{romlist}

As in Section \ref{caseB}, we need to show that, after contracting each pentagon to a vertex, 
the ensuing non-backtracking walk $\pi'(w)$ is a SAW on $G'$. For any subwalk 
$\nu$ of $\pi'(w)$, it may be checked (see (iv) above,
as in the proof of Section \ref{caseB}),
that its numbers of right and left turns differ by at most $1$. 
By Lemma \ref{l96}(c) or directly, $\nu$ cannot form a cycle. 
Hence $\pi'(w)$ (and therefore $\pi(w)$ also) is a SAW, and the proof is complete.

\section{Groups with two or more ends}\label{sec:ends}

\subsection{Groups and ends}
The number of \emph{ends} of a connected graph $G$ is the supremum over its finite subgraphs
$H$ of the number of infinite components that remain after removal of $H$. 
We recall from \cite[Prop.\ 6.2]{BM91} that the number of ends of an infinite transitive graph
is invariably $1$, $2$,  or $\oo$. Moreover, a two-ended (\resp, $\oo$-ended)
graph is necessarily amenable (\resp, non-amenable).
The number of ends of a finitely generated group is the number of 
ends of any of its Cayley graphs.

We present two principal theorems in this section concerning
Cayley graphs of $2$-ended and $\oo$-ended groups, with further results
in Section \ref{sec:mult}. Theorems \ref{thm:2endgp} and \ref{thm:minimal}
are proved, \resp, in Sections \ref{sec:pf2end} and \ref{sec:pf-infty}.
As in \cite{GrL3}, all Cayley graphs in this paper are in their \emph{simple} form, that is,
multiple edges are allowed to coalesce. 

\begin{theorem}\label{thm:2endgp}
Let $\Ga=\langle S\mid R\rangle$ be a finitely
presented group with  two ends. Any Cayley graph $G$ of
$\Ga$ with degree $3$ or more satisfies $\mu(G) \ge \phi$.  
\end{theorem}

We have only partial results for $\infty$-ended Cayley graphs of 
finitely generated groups $\Ga=\langle S \mid R\rangle$, as given in Theorem \ref{thm:minimal}.
As usual, we consider only finite generator sets $S$ with $\id \notin S$
and which are \emph{symmetric} in that $S=S^{-1}$. 
A generator set $S$ is called \emph{minimal} 
if no proper subset is a generator set.
Stallings \cite{Stall1,Stall2} proved that a group with two or  infinitely many ends
is either an HNN extension or an amalgamated product (see Section \ref{sec:mult}
for further details, and for an explanation of the above terms).
The proof of Theorem \ref{thm:minimal}  is found  in Section \ref{sec:pf-infty}.

\begin{theorem}\label{thm:minimal}
Let $\Gamma$ be a finitely generated group with infinitely many ends, and 
let $m$ $(\ge 3)$ be the minimum cardinality of a generator set.
\begin{letlist}
\item If $m \ge 4$, then $\mu(G) \ge \sqrt 3$ $(>\phi)$ for any Cayley graph $G$ of $\Ga$.
\item Suppose $m=3$ and $\Ga$ is an HNN extension. Any Cayley graph $G$ of $\Ga$ satisfies
$\mu(G) \ge\phi$. 
\item Suppose $\Ga=H *_C K$ is an amalgamated product
with generator set $S \subseteq H\cup K$ satisfying $|S|=3$.
There exists a 
minimal generator set $S' \subseteq H\cup K$
with $|S'|=3$ whose Cayley graph $G'$ satisfies $\mu(G') \ge \phi$.
\item If $\Ga=H*_CK$ is an amalgamated product,  there exists
a minimal generator set $S$ whose Cayley graph $G$ satisfies $\mu(G) \ge \phi$.
\end{letlist}
\end{theorem}

Further results are presented in 
Theorem \ref{moe},
and will be used in the proof of Theorem \ref{thm:minimal}.
We do not know whether the above two theorems can be extended 
to multiply ended transitive \emph{graphs}. Indeed, we have no example of a
$2$-ended, transitive, cubic graph that is not a Cayley graph (see \cite{Wat}). 

\subsection{Proof of Theorem \ref{thm:2endgp}}\label{sec:pf2end}
We are grateful to Anton Malyshev 
for his permission to present his ideas in this proof.
Let $\Ga$ be as in the statement of the theorem, and recall from 
\cite[Thm 1.6]{Co70} (see also \cite{Ho72,Ox72,Stall1}) 
that there exists $\beta \in \Ga$ with infinite order
such that the infinite cyclic subgroup $\sH:=\langle \beta\rangle$ 
of $\Ga$ has finite index, and $\be$ preserves the ends of $\Ga$. By Poincar\'e's theorem for subgroups, we may 
choose $\beta$ such that $\sH \normal \Ga$.  We write $\om_1$ 
for the end of $\Ga$ containing the ray $\{\be^k\id: k =1,2,\dots\}$, and $\om_0$ for its other end. 

Let $F:\sH\to\ZZ$ be given by $F(\be^n)=n$, and let
$G$ be a locally finite Cayley graph of $\Ga$. 
By \cite[Thm 3.4(ii)]{GL-Cayley},
there exists a harmonic, \hdi\ function $h:\Ga\to\RR$ that
agrees with $F$ on $\sH$.

Let $g$ be a harmonic function on $G$.
For an edge $\vec e=[u,v\rangle$ of $G$ endowed
with an orientation from $u$ to $v$, we write $\De g(\vec e) = g(v)-g(u)$.
A \emph{cut} of $G$ is a finite set of edges that separates the two ends of $G$; a cut is \emph{minimal}
if no strict subset is a cut. The ($g$-)\emph{size} of a cut $C$ is given as the aggregate $g$-flow across $C$, that is,
\begin{equation*}
s_C(g) = \sum_{\vec e \in C} \De g(\vec e),
\end{equation*}
where the sum is over all edges in $C$ oriented such that initial vertex (\resp, final vertex) 
of each edge is connected in $G \smi C$ to  $\om_0$ (\resp, $\om_1$).
Here are two observations of which the first is standard.
(See \cite[Chap.\ 1]{G-pgs} for a general account of flows and 
electrical networks.)
\begin{letlist}
\item 
Since $g$ is assumed harmonic, $s_C(g)$ is constant for all minimal cuts $C$.
(\emph{Outline proof.} 
Let $C_1$, $C_2$  be minimal cuts, and let
$D$ be a minimal cut such that each
$C_i$ lies in the connected component of $G \smi D$ containing $\om_1$. 
Let $i\in\{1,2\}$ and let $F_i$ be the 
union of bounded components of
$G \smi (C_i\cup D)$. By summing $\De g(\vec e)$ over all oriented edges $\vec e=[u,v\rangle$ of $G$ with $u \in F_i$
(so that each undirected edge of $F_i$ appears twice, once with each orientation),
and using the fact that $g$ is harmonic, we find that $s_{C_i}(g)=s_D(g)$. The claim follows.) 
We write $s(g):= s_C(g)$ for the \emph{size} of $g$, that is, 
the common $g$-size of all minimal cuts $C$.

\item 
We have that $s(h)\ne 0$. 
(\emph{Outline proof}. 
Assume $s(h)=0$, so that $s_C(h)=0$ for all minimal cuts $C$. Let $C$ be a minimal cut and find $k \ge 1$ such that
$\id$ and $\beta$ lie in the same bounded component $F$ of $G\smi(\be^{-k}C \cup \be^k C)$. Let $F'$ be obtained from $G$ by identifying all endpoints of
edges of $\be^{-k}C$ which do not lie in $F$ (\resp, 
all endpoints of edges of $\be^k C_2$ not lying in $F$) as  a single composite vertex
$c^-$ (\resp, $c^+$). Since $s_{\be^{-k}C}(h)=s_{\be^kC}(h)=0$, 
the function $h$ is harmonic on $F'$, and in addition the total $h$-flow 
through $F'$ from $c^-$ to $c^+$ is zero.
It follows that $h$ is constant on $F$. In particular, $h(\id)=h(\beta)$,
a contradiction.) 
\end{letlist}

We now develop the argument of Proposition \ref{prop1}. 
Let $\{\kappa_i : i\in I\}$ be a set of
representatives of the cosets of $\sH$, so that $\Ga/\sH=\{\kappa_i\sH: i \in I\}$ and $|I|<\oo$.
For $\kappa \in \Ga$, we write $\sgn(\kappa)=1$ (\resp, $\sgn(\kappa)=-1$)
if the ends of $\Ga$ are $\kappa$-invariant (\resp, the ends are swapped under $\kappa$). 
Note that
\begin{equation}\label{eq:swap}
s(\kappa h) = \sgn(\kappa)s(h)
\end{equation}
where $\kappa h(\a):= h(\kappa\a)$ for $\a\in\Ga$. 

Let $g:\Ga\to\RR$ be given by
\begin{equation}\label{eq:swap2}
g(\a) = \sum_{i\in I} \sgn(\kappa_i)h(\kappa_i\a), \qq \a\in\Ga.
\end{equation}
Since $g$ is a sum of harmonic functions, it is harmonic. 
Furthermore,  (as in the proof of Proposition \ref{prop1}),
$g$ is $\Ga$-\emph{skew}-difference-invariant in that
\begin{equation}\label{eq:skew}
g(\a v)-g(\a u) = \sgn(\a)[g(v)-g(u)], \qq u,v \in \Ga,\ \a \in \Ga.
\end{equation}
By \eqref{eq:swap} and \eqref{eq:swap2}, $s(g)=|I| s(h)\ne 0$, whence
$g$ is non-constant.

Let $a$, $b$, $c$ denote the values of $g(v)-g(u)$ for  $v\in\pd u$.
By \eqref{eq:skew}, $a$, $b$, $c$ are independent of the
choice of $u$ up to negation, and, since $g$ is harmonic,
$a + b + c = 0$. By re-scaling and re-labelling where necessary,  since $g$ is non-constant
we may assume $|a|,|b|\le c = 1$. The directed edge $\vec e=[u,v\rangle$ is \emph{labelled} 
with the corresponding letter (with ambiguities handled as below), and is allocated weight $\De g(\vec e)$. 
Thus, a directed edge labelled $d$ has weight $\pm d$.

A SAW $\pi=(\pi_0,\pi_1,\dots,\pi_n)$ is called \emph{maximal} if
$g(\pi_k) < g(\pi_n)$ for $k < n$. We shall construct a family of maximal SAWs 
$\pi$ of sufficient cardinality to yield the claim. Choose $(\pi_0,\pi_1)$ such that
$g(\pi_1)=g(\pi_0)+1$.
There are three possibilities for the vector $(a,b,c)$.
\begin{letlist}
\item  Suppose $(a, b, c) = (0,-1, 1)$. 
For $n \ge 1$, a maximal SAW $\pi=(\pi_0,\pi_1,\dots,\pi_n)$ can be extended to two
distinct maximal SAWs by adding either (i)  the directed edge $[\pi_n,w\rangle$ 
with weight $1$, 
or (ii) the  directed edge $[\pi_n,w\rangle$ with weight $0$, 
followed by the edge $[w,x\rangle$ with 
weight $1$.
The number  $w_n$ of such walks of length $n$ from a given starting point
satisfies $w_{n}=w_{n-1}+w_{n-2}$, whence $\mu \ge \phi$.

\item Suppose $(a, b, c) = (-\tfrac12,-\tfrac12, 1)$.
Since there are no odd cycles comprising 
only edges with weight $\pm\tfrac12$, 
the labels of such edges, $\langle u,v\rangle$ say, 
may be arranged in such a way 
that $[u,v\rangle$ and $[v,u\rangle$ receive the same label.  
A maximal SAW $\pi=(\pi_0,\pi_1,\dots,\pi_n)$ that ends with a $c$-edge can 
be extended by following sequences of
additional directed edges labelled one of $ac$, $bc$, $abac$, $babc$,
thus creating a new walk denoted $\pi'$. 
Since $\pi$ is maximal, we have that $h(\pi_n)-h(\pi_{n-1})=1$.
By tracking the signs of the weights of any  additional edge, we see 
that any such $\pi'$ is both self-avoiding and maximal.  
The number $w_n$ of such SAWs with length $n$ from a given starting point 
satisfies $w_n=2w_{n-2}+2w_{n-4}$.
Therefore, $\lim_{n\to\oo} w_n^{1/n}$ equals the root in $[1,\oo)$
of the equation $x^{4} = 2(x^{2}+1)$,
namely $x = \sqrt{1+\sqrt 3}>\phi$.

\item
Suppose $b<a<0$, $-a-b=c=1$. 
There are no
cycles comprising only directed edges labelled either $a$ or $b$.
A maximal SAW $\pi=(\pi_0,\pi_1,\dots,\pi_n)$ that ends with a $c$-edge can 
be extended by following edges labelled either (i) $ac$, or 
(ii) $bc$, $babc$, $bababc$, and so on; 
any such extension results in a maximal SAW, as in case (b) above. 
The number $w_n$ of such SAWs with an
odd length $n$ from a given starting point 
satisfies $w_{n}= 2w_{n-2} + w_{n-4} + w_{n-6} + \cdots+w_1$. 
It is easily checked that $w_n \ge C\phi^n$, as required.
\end{letlist}
The proof is complete.

\subsection{Multiply ended graphs}\label{sec:mult}

Let $\Ga$ be an infinite, finitely generated group. By 
 Stalling's splitting theorem (see \cite{Stall1,Stall2}),
$\Ga$ has two or more ends if and only if one of the following two properties holds.
\begin{romlist}
\item  
$\Gamma=\langle S,t \mid R,\,t^{-1}C_1t=C_2\rangle$ is an HNN extension,
where $H=\langle S\mid R\rangle$ is
a presentation of the group $H$, $|S|<\oo$,
$C_1$ and $C_2$ are isomorphic finite subgroups of $H$, 
and  $t$ is a new symbol.
\item $\Ga=H*_C K$ is a free product with amalgamation, 
where $H$, $K$ are groups, and $C\ne H,K$ is a finite group.
\end{romlist}
A further characterization of $\oo$-ended groups is provided in \cite{Stall2}
(see also \cite[Sect.\ 1.3]{Cannon}).

Part (i) above may be taken as the definition of an HNN extension, 
named as an acronym of the authors of \cite{hnn}. 
The amalgamated product $H*_C K$ of part (ii)  is 
an extension of the notion of the free product, and is defined as follows.
Let $H$, $K$, $C$ be groups, 
and let $\phi:C\to H$, $\psi: C \to K$ be injective homomorphisms.
Let $N$ be the smallest normal subgroup of the free product $H * K$ containing all
elements of the form $\phi(c)\psi(c)^{-1}$, $c \in C$. Then  $H*_C K$ is
defined to be the quotient group $(H*K)/N$.
 
Readers are referred to \cite{BMR,LS77,MKS}
for the background and properties of amalgamated products.
Although the group $C$, in (ii), is not required to
be a subgroup of either $H$ or $K$, when we 
speak of $C$ as such a subgroup we mean the image of $C$ under the corresponding map 
($\phi$ or $\psi$, as appropriate).   
We will generally assume that $C \ne H,K$. 
We next remind the reader of the normal form theorem (Theorem \ref{l23}) for amalgamated products, 
and then we summarise the results of this section in Theorem \ref{moe}. 

\begin{theorem}[Normal form,
{{\cite[Sect.\ 2.2]{BMR}}, \cite[p.\ 187]{LS77}, \cite[Cor.\ 4.4.1]{MKS}}] \label{l23}
\mbox{}
\begin{letlist}
\item Every  $g \in H*_C K$ can be written in the \emph{reduced form}
$g=cv_1\cdots v_n$
where $c \in C$, and the  $v_i$ lie in either $H\setminus C$ or $K\setminus C$ and they alternate between 
these two sets. 
The length $l(g):=n$ of $g$ is uniquely determined, and $l(g)=0$ if and only if $g\in C$.
Two such expressions of the form $v_1\cdots v_n$,
$w_1\dots w_n$  represent 
the same element in $H*_C K$ if and only if there exist $c_0\ (=\id),c_2,\dots, c_n\ (=\id)\in C$ such that
$w_k=c_{k-1}v_k c_k^{-1}$.

\item 
Let $A$ (\resp, $B$) be a set of right coset representatives of (the image of) $C$ in $H$ (\resp, $K$),
where the representative of $C$ is $\id$.
Every  $g \in H*_{C}K$ can be expressed uniquely in the \emph{normal form}
$g=cx_1\cdots x_n$
where $c \in C$, and the  $x_i$ lie in either $A$ or $B$, and they alternate between 
these two sets. 
\end{letlist}
\end{theorem}

Here are the results of this section.

\begin{theorem}\label{moe}
\mbox{}
\begin{romlist}
\item 
Let $\Gamma=\langle S,t \mid R,\,t^{-1}C_1t=C_2\rangle$ be an HNN extension as above. 
Any locally finite Cayley graph $G$ of $\Gamma$ admits a group height function
(see \cite{GL-Cayley}).
If such $G$ is cubic, then $\mu(G) \ge \phi$.
\item 
Let $\Gamma=H *_C K$ be an amalgamated product as above.
\begin{letlist}
\item Suppose $C=\{\id\}$.
Let $S \subseteq H\cup K$
be a generator set of $\Ga$ satisfying $|S|\ge 3$.
The corresponding Cayley graph $G$ satisfies $\mu(G)\geq \phi$. 
\item Suppose $C\neq \{\id\}$. Any generator set $S$ satisfying
both 
\begin{numlist} 
\item $S\cap C\neq \es$, $|S|\geq 3$, and
\item there exists $s_1\in S$ (\resp, $s_2\in S$)
with a normal form beginning with an element of $H\setminus C$
(\resp, an element of $K\setminus C$),
\end{numlist}
generates a Cayley graph $G$ with $\mu(G)\geq \phi$.
\item Suppose $C\neq \{\id\}$ and $C$ is a normal subgroup of both $H$ and $K$. 
Any generator set $S$ satisfying $S\cap C\neq \es$ 
generates a Cayley graph $G$ with $\mu(G)\geq \phi$.
\end{letlist}
\end{romlist}
\end{theorem}

\begin{proof}[Proof of Theorem \ref{moe}(i)]
Let $H=\langle S \mid R\rangle$, and 
let $h:\Gamma\to\ZZ$ be the unique function satisfying  $h(\id)=0$ and, for $\g\in \Gamma$,
\begin{align*}
h( \g t)-h(\g) &=1,\\
h(\g s)-h(\g)&=0.\qq  s\in S.
\end{align*}
By \cite[Thm 4.1]{GL-Cayley} applied to the unit vector
in $\ZZ^{S\cup\{t\}}$ with $1$ in the entry labelled $t$, 
$h$ is a group height function (and hence a transitive \ghf) on any locally finite Cayley graph of $\Gamma$. 
When $G$ is cubic, the inequality $\mu(G)\ge \phi$ follows by 
Theorem \ref{thm:1}(b).
\end{proof}

We turn to the proof of Theorem \ref{moe}(ii). 
By \cite[Thm 1]{GL-Comb}, 
we have $\mu(G)\geq \sqrt{3}>\phi$ if the generator-set $S$
satisfies $|S|\ge 4$. 
We may, therefore, assume henceforth that $|S|=3$.
It is straightforward to check the following lemma.

\begin{lemma}\label{l94}
Let $\Gamma$ be a finitely generated group with two or more ends. 
Let $S=\{s,s_1,s_2\}$ be a generator set whose
Cayley graph $G$ is cubic. Then,
subject to possible permutation of the generators, exactly one of the following holds.
\begin{Alist}
\item $s^2=s_1^2=s_2^2=\id$.
\item $s^2=s_1s_2=\id$.
\end{Alist}
\end{lemma}

\begin{proof}[Proof of Theorem \ref{moe}(ii)(a) when $|S|=3$]
When $C=\{\id\}$, $\Ga$ is the free product of $H$ and $K$. 
In this case, $S\cap H$ (\resp, $S \cap K$) generates $H$ (\resp, $K$).
To see this, let $h \in H \smi \{\id\}$. If $h$ is not a word in the alphabet $S \cap H$, 
then it has a shortest representation as a product (of elements in $S$) including at least one
element of $S \cap H$ and one of $S \cap K$. This contradicts Theorem \ref{l23}(a).

Since $S$ is symmetric, without loss of generality we may write $S=\{s,s_1,s_2\}$ with $s\in H$, 
$s_1,s_2\in K$, and either A or B of Lemma \ref{l94} holds.
It suffices to construct  an injection from $\WW_n$ into the set of $n$-step SAWs on $G$ starting from $\id$.  
Let $\langle x,y\rangle$ be an edge of $G$,
so that $x^{-1}y\in S$. When endowed with an orientation, 
the now  oriented edge $[x,y\rangle$
is said to be \emph{labelled} by the generator $x^{-1}y$.

\noindent
{\bf Assume A of Lemma \ref{l94} holds.} Let $w \in \WW_n$, and construct
a SAW $\pi=(\pi_0,\pi_1,\dots,\pi_n)$ on $G$ as follows. We set $\pi_0=\id$, $\pi_1=s$, and
we iterate the following construction for $k=2,3,\dots,n$.
\begin{numlist}
\item If $w_k= $ V, the $k$th edge of $\pi$ is that labelled $s_2$.

\item If $(w_{k-1}w_k) = $ (HH), the $k$th edge is that labelled $s$ (\resp, $s_1$) if the $(k-1)$th edge
is labelled $s_1$ (\resp, $s$).

\item  If $(w_{k-1}w_k) = $ (VH), the $k$th edge is that labelled $s$.
\end{numlist}
The outcome may be expressed in the form $\pi=p_1s_2p_2s_2\cdots$ where each 
$p_i=ss_1ss_1\cdots$ is
a word of alternating $s$ and $s_1$. 
The letters in $\pi$ alternate between the two sets
$H\smi \{\id\}$ and $K\smi \{\id\}$
with the possible exception of isolated appearances of $s_1s_2$,
each of which is in $K \smi \{\id\}$. Now $s_1s_2\ne \id$, whence $s_1s_2\in K\smi \{\id\}$.
By Theorem \ref{l23}(a), the map $w \mapsto \pi$ is an injection.

\noindent
{\bf Assume B of Lemma \ref{l94} holds.}
Let $G_n$ be the Cayley graph of $\ZZ_2\,*\,\ZZ_n$ for $3\le n < \oo$.
We have that $H\cong \ZZ_2$, and the Cayley graph of $K$ is a cycle of length at least $4$. Therefore, $G$ is isomorphic to $G_n$
for some $n \ge 4$. 
The exact value $\mu(G)$ may be deduced from \cite[Thm 3.3]{Gilch}, 
but it suffices here to note that
\begin{equation}\label{eq:fisher3}
\mu(G)\geq \mu(G_3)>\phi.
\end{equation}
The above strict inequality holds since $G_3$ is the graph obtained from the cubic tree by the Fisher 
transformation of \cite{GrL2} (see 
item H of Section \ref{sec:exs}).
\end{proof}

\begin{proof} [Proof of Theorem \ref{moe}(ii)(b) when $|S|=3$]
Let $\Gamma$, $G$, $S$, $s_1$, $s_2$ be as given, and $s \in S\cap C$. We  may write $S=\{s,s_1,s_2\}$, and either A or B of Lemma \ref{l94} holds.
By assumption 2, we have that $s_1,s_2 \notin C$;
since $s^{-1}\in C$, it follows that $s^2=\id$.

Under A, the normal form of $s_1$ 
(\resp, $s_2$) begins and ends with elements of $H \smi C$ (\resp, $K\smi C$).
Under B, the normal form of $s_1$ 
(\resp, $s_2$) ends with an element of $K \smi C$ (\resp, $H\smi C$).
 
Let $w\in \WW_n$, and construct a SAW $\pi$ on $G$ as follows.
We set $\pi_0=\id$, $\pi_1=s_1$, and we iterate the following for $k=2,3,\dots,n$.
\begin{numlist}
\item Suppose $w_k= $ V. The $k$th edge of $\pi$ is that labelled $s$.

\item Suppose $(w_{k-1}w_k) = $ (HH).
The $k$th edge is that labelled $s_1$ (\resp, $s_2$) if the $(k-1)$th edge
is labelled by the member of $\{s_1,s_2\}$ whose normal form 
ends with an element of $K\setminus C$ (\resp, $H\setminus C$).

\item  Suppose $(w_{k-1}w_k) = $ (VH).
The $k$th edge is that labelled $s_1$ (\resp, $s_2$) if the $(k-2)$th step 
is labelled  by the member of $\{s_1,s_2\}$ whose normal form ends with an 
element of $K\setminus C$ (\resp, $H\setminus C$).
\end{numlist}
We claim that the resulting $\pi$ is a SAW. If not, there exists a representation of
the identity of the form 
\begin{equation}\label{eq:identity}
\id = p_1sp_2s \cdots s p_r,
\end{equation}
where $r \ge 1$, and each $p_i$ is a non-empty alternating product of elements of $H\smi C$ and $K\smi C$
such that $p_1p_2\cdots p_r$ is such a product also, with some aggregate length $L\ge 1$ 
(we allow also that $p_1$ and/or $p_r$ may equal $\id$). 
Each $s$ in \eqref{eq:identity} lies in $C$, and we may move these
members of $C$ to the left end of the product by the following procedure. 
Consider a consecutive pair of elements in \eqref{eq:identity}
of the form $p_i s$, with $p_i\in H\smi C$ (\resp, $p_i\in  K\smi C$).
Now, $p_i s$ lies in some right coset of $C$ in $H$ (\resp, in $K$),
whence $p_i s = c_i a_i$ for some $c_i\in C$ and  $a_i \in A$ 
(\resp, $a_i \in B$), where we have used the 
notation of Theorem \ref{l23}(b).  Replacing $p_i s$ by $c_ia_i$,
and  iterating this procedure, we obtain a normal form
$c' v_1v_2\cdots v_L$, which cannot equal the identity since $L \ge 1$.  
This contradicts \eqref{eq:identity}, and the claim of part (b) follows. 
\end{proof}

\begin{proof}[Proof of Theorem \ref{moe}(ii)(c) when $|S|=3$]
We write $S=\{s,s_1,s_2\}$. 
Clearly, $|S\cap C|\leq 2$, since $C$ is a 
proper subgroup of both $H$ and $K$.

Assume that $|S\cap C|=2$, and let $\{s_1,s_2\}=S\cap C$ and $\{s\}=S\setminus C$. Since $s^{-1}\notin C$, it follows
by Lemma \ref{l94} that $s^2=\id$. 
Since $C$ is a normal subgroup of both $H$ and $K$, we have
by Theorem \ref{l23} that $\a C\a^{-1}=C$ for $\a\in \Ga$.
Since $S$ generates $\Ga$, every $g\in \Ga$ may be expressed as a word in the alphabet $\{s,s_1,s_2\}$,
and hence in the form $g=c_1 s c_2 s\cdots s c_r$ with $c_i\in C$. By the normality of $C$,
$g=cs^k$ for some $c \in C$, $k \in \NN$. However, $s^2=\id$,
so that there are only finitely many choices for $g$, a contradiction.

Therefore, we have $|S\cap C|=1$, and we write $\{s\}=S\cap C$ and $\{s_1,s_2\}=S\smi C$. 
Either A or B of Lemma \ref{l94} holds.

If one of $\{s_1,s_2\}$ has a normal form starting with an element in 
$H\setminus C$, and the other has a normal form starting with an 
element in $K\setminus C$, then the claim follows  by Theorem \ref{moe}(ii)(b). 
For the remaining case, we may assume without loss of 
generality that the normal forms of both $s_1$ and $s_2$ 
start with elements in $H \smi C$.  It follows that, under either A and B, 
both normal forms end in $H\smi C$.

Here is an intermediate lemma, proved later in this section.

\begin{lemma}\label{lem:tt}
For $j\in \NN$,
\begin{equation}\label{eq:claim2}
\begin{aligned}
\text{if A holds,}&\q (s_1s_2)^j, (s_1s_2)^{j-1}s_1, (s_2s_1)^{j-1} s_2\notin C,\\
\text{if B holds,}&\q s_1^j\notin C.
\end{aligned}
\end{equation}
\end{lemma}

We shall construct an injection from  the set
$\WW_n$ into the set of $n$-step SAWs on $G$ from $\id$. 
For $w\in \WW_n$, we construct a SAW $\pi$ on $G$ with $\pi_0=\id$, $\pi_1=s_1$ as follows.
\begin{numlist}
\item Each letter V in $w$ corresponds to an edge in $\pi$ with label $s$.
\item Assume A holds. 
The letters H in $w$ correspond to the 
elements of the sequence $(s_1,s_2,s_1,s_2,\dots)$,
in order. That is, for $k\ge 1$,
the $(2k-1)$th (\resp, $(2k)$th) occurrence of H
corresponds to $s_1$ (\resp, $s_2$).
 
\item Assume B holds. The letters H in $w$ correspond to edges labelled $s_1$.
\end{numlist}
We show next that the resulting walks are self-avoiding.

\noindent
{\bf Assume B holds.} 
If one of the corresponding walks fails to be self-avoiding, there exists a representation of the identity as
\begin{equation}\label{eq:1234}
\id = s_1^{k_1}ss_1^{k_2}s \cdots s s_1^{k_r},
\end{equation}
where $r \ge 1$, $k_1,k_r \in \NN\cup\{0\}$, $k_i \in \NN$ for $2\le i < r$,
and $K=k_1+\dots+k_r\ge 1$.
Since $C$ is normal,  we have $\id = c s_1^K$ for some $c\in C$.
This contradicts \eqref{eq:claim2}, and we deduce that each such $\pi$ is self-avoiding.

\noindent
{\bf Assume A holds.}
The above argument remains valid with adjusted \eqref{eq:1234}, and yields that 
$\id = c t$ for some $c\in C$ and 
$$
t \in \bigl\{ (s_1s_2)^j, (s_2s_1)^j,(s_1s_2)^{j-1} s_1,  (s_2s_1)^{j-1} s_2 : j \in \NN\bigr\}.
$$
Therefore, $t=c^{-1}\in C$, in contradiction of \eqref{eq:claim2}.
We deduce that each such $\pi$ is self-avoiding.
\end{proof}

\begin{proof}[Proof of Lemma \ref{lem:tt}]
Let $t_1\in H\setminus C$ and $t_2\in K\setminus C$, 
so that  
\begin{equation*}
l([t_1t_2]^n)=2n, \qq n \in \NN.
\end{equation*}
Since $S$ generates $H*_C K$, we can express $t_1t_2$ as a word in
the alphabet $\{s,s_1,s_2\}$, denoted $t(s,s_1,s_2)$.
Let $\wtilde{t}$ be the word obtained from $t(s,s_1,s_2)$ by 
removing all occurrences of $s$ and using the group relations 
on $S$ to reduce the outcome to a minimal form. 
More precisely, since $s\in C$ and $C$ is normal in $H$ and $K$, every occurrence of $s$ in $t(s,s_1,s_2)$
may be moved leftwards to obtain $t(s,s_1,s_2)=ct'(s_1,s_2)$ for some $c\in C$ and some 
word $t'(s_1,s_2)$. On reducing $t'$ by the group relations
on $S$, we obtain $\wtilde t$, and note that
\begin{equation}\label{eq:new777}
\begin{aligned}
\text{if A holds,}\q&
\text{$\wt t$ is an alternating product of $s_1$ and $s_2$.}\\
\text{if B holds,}\q&\text{$\wt t \in \{s_1^k,s_2^k: k \in \NN\}$.}
\end{aligned}
\end{equation}
Since $\wtilde{t}=c^{-1}t_1t_2$, we have $l(\wtilde{t})=\ell(t_1t_2)=2$.
By the normality of $C$ again, there exists $c'\in C$ such that
$l((\wtilde t)^n) =l(c'[t_1t_2]^n) =2n$.
In particular, by Theorem \ref{l23}(a),
\begin{equation}\label{lt}
(\wtilde t)^n\notin C, \qq  n \in\NN.
\end{equation}

\noindent
{\bf Suppose A holds.} By \eqref{eq:new777},
\begin{equation}
\wtilde t\in
\bigl\{(s_1s_2)^k, (s_2s_1)^k,(s_1s_2)^{k-1} s_1,(s_2s_1)^{k-1} s_2
: k\in \NN\bigr\}.
\label{s1s2}
\end{equation}
If $\wtilde t \in\{ (s_1s_2)^{k-1} s_1, (s_2s_1)^{k-1} s_2: k \in \NN\}$, 
we have $(\wtilde{t})^2=1$, which contradicts \eqref{lt}. Therefore,
\begin{equation*}
\wtilde t\in
\bigl\{(s_1s_2)^k, (s_2s_1)^k: k\in \NN\bigr\}.
\end{equation*}
If $(s_1s_2)^j\in C$ for some $j\in\NN$, then 
\begin{equation*}
(\wtilde t)^j\in \bigl\{ [(s_1s_2)^j]^k, [(s_2s_1)^j]^k: k \in \NN\bigr\}\subseteq C,
\end{equation*}
which contradicts \eqref{lt}. Hence $(s_1s_2)^j\notin C$ for $j\in\NN$, as required. 
Suppose next that $c:=(s_1s_2)^{j-1}s_1\in C$ for some $j \in\NN$.
Since $C$ is a normal subgroup of both $H$ and $K$, we have $s_2cs_2^{-1}=s_2(s_1s_2)^{j}\in C$. 
Therefore, $(s_1s_2)^{2j}\in C$, which contradicts \eqref{lt} as above.
A similar argument holds for the case $c:=(s_2s_1)^{j-1}s_2$.
The first statement of \eqref{eq:claim2} is proved.

\noindent
{\bf Suppose B holds.} A similar argument is valid
by \eqref{eq:new777}, as follows. 
Suppose $\wtilde t=s_1^k$ (a similar argument holds in the other case, using the fact
that $s_1s_2=\id$).
If $s_1^j\in C$ for some $j\in \NN$, then 
$(\wtilde t)^j=(s_1^j)^k\in C$,
in contradiction of \eqref{lt}. The second statement of \eqref{eq:claim2} follows.
\end{proof}

\subsection{Proof of Theorem \ref{thm:minimal}}\label{sec:pf-infty}

Since $\Ga$ has infinitely many ends, any generator set $S$ has cardinality 3 or more. 
In particular, $m \ge 3$. Part (a) follows by \eqref{eq:lowerbound}.
Part (b) follows by Theorem \ref{moe}(i). 
We turn to part (c),
and assume henceforth that $\Ga=H*_C K$ is an amalgamated product as in Section \ref{sec:mult}.
  
Let $S\subseteq H\cup K$ be a generator set of $\Ga$ with $|S|=3$. 
We may assume $S$ is minimal, as follows. If $S$ is not minimal, it has
a proper subset which is a generator set, in contradiction of the fact that $|S|=m=3$.
 
Since $C$ is a proper subset of both $H$ and $K$, there exist
$s_1\in S\cap (H\setminus C)$ and $s_2\in S\cap (K\setminus C)$. 
Let $s\in S\smi\{s_1,s_2\}$ and, without loss of generality, assume $s\in H$. 
If $s\in C$, the inequality $\mu\ge\phi$  follows by Theorem \ref{moe}(ii)(b). 
We may, therefore, assume henceforth that $s\in H\setminus C$, 
so that 
\begin{equation}\label{eq:hk}
s,s_1\in H\smi C, \q s_2\in K\smi C.
\end{equation}
By Lemma \ref{l94}, 
one of the following occurs.
\begin{Alist}
\item $s^2=s_1^2=s_2^2=\id$.
\item $s_2^2=ss_1=\id$.
\end{Alist}

\noindent
{\bf Assume A holds}.
\begin{letlist}
\item
Suppose $ss_1\in C$. 
If $ss_1=s_1s$, then $S':=\{ss_1,s_1,s_2\}$ is a minimal generator set 
with Cayley graph $G'$ satisfying $\mu(G')\ge\phi$ by Theorem \ref{moe}(ii)(b). 

Suppose $ss_1\ne s_1s$, and let $\om$ be the order of $ss_1$, that is, the least $k$ such that $(ss_1)^k=\id$.
Since $ss_1\ne s_1s$, we have that $\om \ge 3$, and hence
\begin{equation}\label{eq:new3}
\begin{aligned}
s,ss_1s,ss_1ss_1s &\text{ are distinct elements of } H \smi C, \text{ and}\\
s_1,s_1ss_1, s_1ss_1ss_1 &\text{ are distinct elements of } H\smi C.
\end{aligned}
\end{equation}
Let $\Pi$ be the set of finite labelled walks on the Cayley graph $G$ of $S$ starting at $\id$ and satisfying:
\begin{numlist}
\item the first edge is labelled $s_2$,
\item between any two consecutive appearances of edges 
labelled $s_2$, there appears one of the six words
in \eqref{eq:new3}, and nothing further,
\item after the final appearance of $s_2$, there appears one of the words in \eqref{eq:new3}.
\end{numlist} 

We claim that members of $\Pi$ are SAWs on $G$, and we prove this next.
A walk $\pi\in \Pi$ is a word in the alphabet $S$ with the form
$\pi=s_2a_1s_2a_2s_2\cdots s_2a_r$ where $r \ge 1$ and $a_1,a_2,\dots,a_r\in H\smi C$.
Now, $\pi$ is a SAW if and only if no subword of $\pi$ equals the identity $\id$.
Any subword containing the letter $s_2$ has length (in the sense of Theorem \ref{l23}) at least one, 
and is therefore not the identity.  Let $\nu$ be a subword not containing $s_2$,
so that $\nu$ is a subword of some $a_i$. By inspection of \eqref{eq:new3}, and the fact that $ss_1$ has
order three or more,  we have that $\nu\ne \id$.

The generating function \eqref{eq:genfn} corresponding to
the set $\Pi$ is
$$
Z(\zeta) = \sum_{k=1}^\oo f(\zeta)^k \q\text{where} \q f(\zeta)=\zeta(2\zeta+2\zeta^3+2\zeta^5).
$$
By a simple calculation, $f(1/\phi)>1$, whence
\begin{equation}\label{eq:new6}
Z(1/\phi) =\oo.
\end{equation}
Therefore, $\mu(G)  \ge \phi$.

\item
Suppose $ss_1 \in H\setminus C$. We construct an 
injection from $\WW_n$ into the set of $n$-step SAWs on $G$ from $\id$
as follows.
Let $w \in \WW_n$, and let $\pi$ denote  the following walk on $G$.  Set $\pi_0=\id$, $\pi_1=s_2$.
\begin{numlist}
\item
At each occurrence of V in $w$, $\pi$ traverses the edge labelled $s_1$.
\item
Any run of the form H$^r$ in $w$ corresponds to a walk $s_2,s_2s,s_2ss_2,s_2ss_2s,\dots$ of 
length $r$ in $\pi$.
\end{numlist}
The resulting $\pi$ traverses the  edges of $G$ in the manner of a word of the form
$\a=(a_1s_1a_2s_1\cdots s_1a_r)$ where each $a_i$ is a word starting with $s_2$ and 
alternating $s$ and $s_2$
(we allow $a_r$ to be empty). By \eqref{eq:hk}, each $a_i$ is in the reduced form of Theorem \ref{l23}(a).
At each occurrence of $s_1$ in $\a$, there may be a consecutive appearance of generators in $H \smi C$ 
taking the form $ss_1$. 
At each such instance, we may group $ss_1$ as a single element of $H \smi C$, thus 
obtaining a reduced form for $\a$.

If $\pi$ is not self-avoiding, some non-trivial subword of $\a$ equals the identity $\id$. By
Theorem \ref{l23}, this subword must have length $0$, which cannot occur. Therefore, $\pi$ is a SAW.
\end{letlist}

\noindent
{\bf Assume B holds}, so that $s_1^{-1}=s$.
If $C=\{\id\}$,  the claim follows by Theorem \ref{moe}(ii)(a). Assume $C\neq\{\id\}$.
Consider the minimal generator set $\wtilde{S}=\{s_2,u:=s_2s_1,v:=ss_2\}$
with corresponding Cayley graph $\wtilde{G}$. 
We  construct an injection $f$ 
from $\WW_n$ into the set of $n$-step SAWs on 
$\wtilde{G}$ from $\id$, as follows. Let $w\in\WW_n$
and construct $\pi=f(w)$ as follows. Set $\pi_0=\id$, $\pi_1=u$,
and let $k \ge 2$.
\begin{numlist}

\item If $w_k=$ V, the $k$th edge
of $\pi$  is labelled $s_2$.

\item If $w_k=$ H, the $k$th edge of $\pi$ lies in $\{u,v\}$.

\begin{romlist}
\item If $(w_{k-1}w_k)=$ (HH), the $k$th edge of $\pi$ has the same 
label as the $(k-1)$th.

\item If $(w_{k-1}w_k)=$ (VH), the $k$th edge
of $\pi$ is labelled as the inverse of that of the $(k-2)$th.
\end{romlist}
\end{numlist}
The resulting $\pi$ has the form of the word 
\begin{equation}\label{eq:new778}
\a=(u^{k_1}s_2 v^{k_2}s_2\cdots s_2 x^{k_r})
\end{equation}
where $r\ge 2$,
$k_i\ge 1$ (we allow $k_r=0$), the powers of $u$ and $v$ alternate,  and $x \in \{u,v\}$ as appropriate.
The walk $\pi$ is the image $\pi=f(w)$ where 
$w=$ H$^{k_1}$VH$^{k_2}$V$\cdots $VH$^{k_r}$.
By considering the various possibilities (as follows), 
we obtain that every non-trivial subword
of $\a$ has non-zero length, and hence $\pi$ is a SAW.  

Here is a brief amplification of the last stage. Suppose
the walk $\pi$ contains some cycle. Then
the word $\a$
of \eqref{eq:new778} contains a subword of the form 
$\beta=t_1^{l_1}s_2t_2^{l_2}s_2 \cdots s_2 x^{l_m}$
that satisfies $\beta=\id$,
where $m \ge 2$, $l_i>0$ (we allow $l_1=0$ and $l_m=0$,
but not both),
$\{t_1,t_2\}=\{u,v\}$ and the powers of $u$ and $v$ alternate, 
and $x\in\{u,v\}$ is 
chosen accordingly. The only
cancellations that can arise in $\beta$ from the group relations on
$S$ (under case B) are of the form $s_2s_2=\id$.
Such a product appears only where either (i) $\beta$ ends with the sequence
$vs_2$ (so that $l_m=0$), or (ii) some $s_2$ is preceded by $v$ 
and followed
by $u$, thus forming the subsequence $vs_2u$. 
At each such occurrence, exactly one cancellation occurs.
The resulting word $\beta'$ (after such cancellations)  
is an alternating product of terms in $H\smi C$ and $K \smi C$,
with strictly positive length. 
(For example, if $\beta=v^{t_1}s_2 u^{t_2}s_2 v^{k_3}s_2$
where $t_1,t_2,t_3>0$, then 
$\beta'=v^{t_1-1}ss_2s_1 u^{t_2-1}s_2v^{t_3-1}s$,
and the last
product, when expanded in terms of the generators $s$, $s_1$, $s_2$,
is in reduced form.) By Theorem \ref{l23}(a),
$\be'\ne\id$, a contradiction. We conclude that $\pi$ is a SAW. 
The proof of part (c) is complete.

Finally, we prove part (d) of the theorem.
If $\Ga$ has a minimal  generator set $S$ satisfying $|S|\geq 4$, the 
corresponding Cayley graph $G$ satisfies $\mu(G)\geq \sqrt{3}>\phi$
by \eqref{eq:lowerbound}.
We may, therefore, assume that \emph{every} minimal generator set of $\Ga$ has cardinality $3$.

By considering the reduced form of Theorem \ref{l23}(a), 
we may find some generator set $S$ satisfying $S\subseteq H\cup K$,  
and, by passing to subsets if necessary, we may assume $S$ is minimal.
By the above, $|S|=3$.
By part (c), there is a minimal generator set $S'
\subseteq H\cup K$ whose Cayley graph $G'$ has $\mu(G')\ge \phi$. 

\begin{remark}\label{rem:10}
Theorem \ref{thm:minimal}(c) falls short of the assertion that $\mu\ge\phi$
\emph{for all} Cayley graphs of amalgamated products with three generators.
There are nevertheless some partial results in this direction. Let $S=\{s,s_1,s_2\}$ 
be a generator set satisfying \eqref{eq:hk}, and let $G$ be the corresponding
Cayley graph. If A holds, and either $ss_1\notin C$, or $ss_1\in C$ and $ss_1\ne s_1s$,
then $\mu(G)\ge \phi$. The proof is given above.
If B holds, one may show that $\mu(G) \ge \phi$ so long as $s_1^2 \notin C$.
The proof is  by construction of an injection from
the set $W_n$ of $n$-step
SAWs on the Cayley graph of the free product 
$\ZZ_2 * \ZZ_3=\langle a,b,b^2\mid a^2,b^3\rangle$ starting at a given vertex, into
the set $\Pi_n$ of $n$-step walks $\pi$ on $G$
with $\pi_0=\id$ and satisfying: $\pi$ can be expressed as 
a word of the form $\a=( a_1s_2a_2s_2\cdots s_2a_r)$
where each $a_i$ lies in $T:=\{s,s^2,s_1,s_1^2\}$ 
(we allow $a_1$ and $a_r$ to be empty).
The details are omitted.
\end{remark}

\section*{Acknowledgements} 
GRG thanks Omer Angel for a conversation regarding Question \ref{qn:big}.
The authors are grateful to Anton Malyshev for permission to include
his unpublished work on the Grigorchuk group and $2$-ended groups,
and to Igor Pak for introducing us to AM and his work.
GRG and ZL express their thanks to an anonymous referee for
fastidious readings of the submitted paper, and for suggestions that have led to improvements. 
This work was supported in part
by the Engineering and Physical Sciences Research Council under grant EP/I03372X/1. 
ZL acknowledges support from the Simons Foundation under  grant $\#$351813
and the National Science Foundation under grant  DMS-1608896.

\providecommand{\bysame}{\leavevmode\hbox to3em{\hrulefill}\thinspace}
\providecommand{\MR}{\relax\ifhmode\unskip\space\fi MR }
\providecommand{\MRhref}[2]{%
  \href{http://www.ams.org/mathscinet-getitem?mr=#1}{#2}
}
\providecommand{\href}[2]{#2}

\end{document}